\documentclass[11pt, reqno]{amsart}
\usepackage{amsfonts}
\usepackage{amssymb}
\usepackage{latexsym}
\usepackage{cite}
\usepackage{bbm}
\usepackage{graphicx}
\usepackage{appendix}
\makeatletter
\renewcommand\section{\@startsection
  {section}{1}{0pt}
  {-3.5ex plus -1ex minus -.2ex}
  {2.3ex plus .2ex}
  {\normalfont\Large\bfseries\raggedright}}
\makeatother

\usepackage[]{enumerate}
\usepackage{verbatim}
\usepackage[]{enumerate}

\usepackage[final]{hyperref}
\hypersetup{colorlinks=true, linkcolor=blue, anchorcolor=blue, citecolor=red, filecolor=blue, menucolor=blue, pagecolor=blue, urlcolor=blue}

\usepackage{bbm}

\renewcommand{\d}{\, \mathrm{d}}

\newcommand{\bigabs}[1]{\bigl\vert #1 \bigr\vert}
\newcommand{\Bigabs}[1]{\Bigl\vert #1 \Bigr\vert}

\newcommand{\norm}[1]{\left\Vert #1 \right\Vert}

\newcommand{\N}{\mathbb{N}}
\newcommand{\Z}{\mathbb{Z}}
\newcommand{\R}{\mathbb{R}}

\newcommand{\angles}[1]{\langle #1 \rangle}

\DeclareMathOperator{\supp}{supp}

\newtheorem{lemma}{Lemma}

\newtheorem{thm}{Theorem}[section]

\newtheorem{lem}[thm]{Lemma}

\theoremstyle{definition}
\newtheorem{definition}{Definition}

\theoremstyle{remark}
\newtheorem{remark}{Remark}

\title{Enhanced lifespan of solutions to Whitham-Boussinesq system with surface tension.
}

\author[M. Kamet] {Madina Kamet}

  \author[A. Tesfahun]{Achenef Tesfahun}

\address{School of Computing and AI \\
Nazarbayev University \\
Qabanbai Batyr Avenue 53 \\
010000 Nur-Sultan \\
Republic of Kazakhstan}

\email{madina.kamet@nu.edu.kz}

\email{achenef@gmail.com}


\begin{document}

\begin{abstract}

We consider
a Whitham--Boussinesq type system with surface tension
arising as asymptotic models in the bi-directional propagation of weakly nonlinear surface
waves in shallow water,  which is characterized by a
nonlinearity parameter $0<\epsilon\le 1$,
 a shallow water parameter $0<\mu\le 1$ and a nonnegative surface tension parameter $\beta$.
 In this paper, we establish well-posedness of the associated Cauchy problem for $\beta>1/3$ on a time scale of order $ ( \beta_\ast^{ 1/4} \mu^{ 1/4} h_0 \epsilon^{-1} )^{4/3}$ in one dimension and of order $(   \beta_\ast^{ 1/4} \mu^{ 1/4} h_0 \epsilon^{-1})^{2-}$ in two dimensions,  where $\beta_\ast= \min(|3\beta-1|, 1/2)$, assuming a non-cavitation condition with a parameter $h_0>0$. 
 In addition, we established  
  well-posedness for the classical Whitham equation with either $\beta=0$ or $\beta>1/3$ on the timescale of order $ ( \beta_\ast^{ 1/4} \mu^{ 1/4}  \epsilon^{-1} )^{4/3}$. These results explicitly capture how the existence time depends on all the parameters involved.
 Our proofs rely on dispersive and Strichartz estimates combined with energy estimates.
\end{abstract}

\maketitle
\section{Introduction}
We consider the fully dispersive Whitham--Boussinesq type system with surface tension
\begin{align}\label{wtbsq} 
    \left\lbrace
    \begin{array}{l}
            \eta_t   +  \mathcal K_{\mu, \beta}(D) \nabla \cdot \mathbf{v} + \epsilon \nabla \cdot (\eta \mathbf{v})=0\\
        \mathbf{v}_t +  \nabla \eta + \frac{\epsilon}{2} \nabla (|\mathbf{v}|^2) =\mathbf{0},
 \end{array}\right.
\end{align}
where $\eta: \R^d \times \R \rightarrow \R$ and $\mathbf{v} : \R^d \times \R \rightarrow \R^d$ for $d\in \{1, 2\}$.
Here, the unknown variable $\eta$ describes the surface elevation and the unknown variable\footnote{
In two-dimensions $\mathbf v =(v_1,  v_2)$ is a vector field, whereas in one dimension  $\mathbf v =v$ is a scalar field.
} $\mathbf{v}$ is related to the fluid velocity, whereas
   $\mathcal K_{\mu, \beta}(D)$ is a nonlocal Fourier multiplier operators with symbol
\begin{equation}
\label{KL-def}
\mathcal K_{\mu, \beta} (\xi) = \mathcal T_\mu(\xi)  \left(1+   \beta \mu|\xi|^2  \right)  \qquad (\xi \in \R^d),
\end{equation}
where $$ \mathcal T_{\mu} (\xi) = \frac{\tanh( \sqrt \mu | \xi|)}{ \sqrt \mu | \xi |} .$$
   Moreover,  $\mu$ and 
$\varepsilon$ are small parameters that characterize the strength of dispersion and nonlinearity, respectively,  while $\beta$ is a nonnegative parameter associated with surface tension.

We complement \eqref{wtbsq} with  initial data 
\begin{equation}
\label{wtbsq-data}
\eta(x, 0)= \eta_0(x),  \qquad  \mathbf v(x, 0)= \mathbf v_0(x) .
\end{equation}

The system \eqref{wtbsq} was proposed in \cite{AMP13, KLP18, L13  }  as approximate models for the study of surface water waves, and provide a two-directional alternative to the well known Whitham equation which in one dimension takes the form
\begin{equation}
\label{wt}
 \eta_t  +  \mathcal K_{\mu, \beta}(D) \eta_x + \epsilon \eta \eta_x  =0\end{equation}
 with  initial data 
\begin{equation}
\label{wt-data}
\eta(x, 0)= \eta_0(x).
\end{equation}
This equation was introduced (for $\beta=0$) by Whitham in \cite{W67} as an alternative to the Korteweg-de Vries (KdV) equation by keeping the exact dispersion of the linearized water waves system in finite depth.
In the case of pure gravity waves
$\beta=0$, it displays several interesting phenomena predicted by Whitham: a solitary
wave regime close to KdV \cite{EGW12}, the existence of a wave of greatest height (Stokes
wave) \cite{EW19}, the existence of shocks \cite{H17}, and modulational instability of steady periodic waves \cite{HJ15, SKCK14}. Note that when surface tension is taken into account $\beta>0$,
the dynamics of \eqref{wt} appears to be completely different (see \cite{KLP18} and the references
therein).
Observe that for low frequencies  $\sqrt \mu |\xi|\ll 1 $,  one has by Taylor expansion,
\begin{equation}\label{K-exp}
    \mathcal K _{\mu, \beta} (\xi)\simeq 1 - \mu  \left( \frac 13-\beta \right)|\xi|^2,
\end{equation}
and so in the formal limit as  $\sqrt \mu |\xi|\rightarrow 0$ (long-wave regime), 
 \eqref{wt} reduces to the KdV equation
 \begin{equation}
\label{Wt-asyp}
 \eta_t  + \eta_x+  \mu  \left( \frac 13-\beta \right)  \eta_{xxx} + \epsilon \eta \eta_x  =0.
\end{equation}

The system \eqref{wtbsq} was 
introduced in \cite{AMP13, L13, MKD2029 } without surface tension ($\beta=0$) and in \cite{KLP18} with surface tension ($\beta>0$). 
In the long-wave regime, it reduces to
  the Boussinesq system
\begin{equation}\label{wtbq-approx}
\left\{
\begin{aligned}
\eta_t +   \nabla \cdot  \mathbf v+ \mu  \left( \frac 13-\beta \right) \Delta  \nabla \cdot  \mathbf v+\epsilon \nabla \cdot (\eta \mathbf v)  & = 0
  \\
\mathbf v_t + \nabla \eta   + \frac{\epsilon}{2} \nabla ( |\mathbf v |^2) &=   0.
\end{aligned}
\right.
\end{equation} 
The later system is a particular member of the (abcd) family of Boussinesq systems derived in \cite{BCS2002} as asymptotic models for water waves in the Boussinesq regime. 
It has been proved in  \cite{SWX2017} that the Cauchy problem \eqref{wtbq-approx}, \eqref{wtbsq-data}
is well-posed when $\beta >1/3$ for initial data $(\eta_0, \mathbf v_0)\in H^s(\mathbb{R}^d)\times H^{s+\frac{1}{2}}(\mathbb{R}^d)$ with $s>d/2$,  and the existence time is shown to be of order $\mathcal O (1/ \sqrt \varepsilon)$ in the long-wave regime $\epsilon=\mu$, while for $\beta <1/3$ the system is known to be ill-posed (see \cite{ABM2019}).

Going back to the full-dispersion systems \eqref{wtbsq}, when surface tension is taken
into account,  the Cauchy problem \eqref{wtbsq}--\eqref{wtbsq-data} with $\epsilon=\mu=1$ was proved to be locally well-posed by Kalisch and Pilod \cite{KP2019} for initial data
$(\eta_0, \mathbf v_0)\in H^s(\mathbb{R}^d)\times H^{s+\frac{1}{2}}(\mathbb{R}^d)$ with  $s>2+\frac{d}{2}$ and $d=1, 2$. Recently, by taking into account the small parameters $\varepsilon$ and $\mu$, Paulsen \cite{P2022} proved long-time well-posedness of \eqref{wtbsq}--\eqref{wtbsq-data} in one dimension with time of existence  $T\sim 1/\epsilon$ for initial data $(\eta_0, v_0)\in X_\mu^s(\mathbb{R})$, where the space $X_{\mu}^s (\R)$ is defined via the norm
\begin{equation*}
\| (f, g)\|^2_{X_{\mu}^s} := \norm{f}_{H^s}^2+\norm{g}_{H^s}^2+ \sqrt \mu  \norm{ |D|^\frac12   g}_{H^s}^2.
\end{equation*}
More recently,  well-posedness of the Whitham equation \eqref{wt}-\eqref{wt-data} and the Whitham-Boussinesq system
\begin{align}\label{wtb} 
    \left\lbrace
    \begin{array}{l}
            \eta_t   + \nabla \cdot \mathbf{v} + \epsilon \nabla \cdot (\eta \mathbf{v})=0,\\
        \mathbf{v}_t +   \mathcal T_{\mu}(D) \nabla \eta + \frac{\epsilon}{2} \nabla (|\mathbf{v}|^2) =\mathbf{0}
 \end{array}\right.
\end{align}
for initial data \eqref{wtbsq-data} 
 were established  by P. Pilod,  S. Selberg, N. Taki and the second author \cite{PSTT25} with time of existence that are dependent on the parameters $\mu$ and $\epsilon$.
More precisely, in \cite{PSTT25} well-posedness of \eqref{wt}-\eqref{wt-data} without surface tension as well as the system \eqref{wtb} in one dimension is proved on a time scale of order $\mu^{(1/4) -}\epsilon^{(-5/4) +}$ and  of \eqref{wtb} in two dimensions  on the time scale of order $\mu^{(1/4)-}\epsilon^{(-3/2)+}$. The  proof is based on energy estimates,  refined Strichartz estimates and compactness argument. The system \eqref{wtb} first appeared in \cite{DDK19} as an approximate model for the study of surface water waves. Well-posedness of this system and its variants have been investigated by several authors, see e.g., \cite{DDT22, D2019, D2020, Dy19,  DST20, E21, T24}.

In this paper,  we establish 
  well-posedness for \eqref{wt}-\eqref{wt-data} with either $\beta=0$ or $\beta>1/3$ on the timescale of order $ ( \beta_\ast^{ 1/4} \mu^{ 1/4}  \epsilon^{-1} )^{4/3}$, 
  where $\beta_\ast= \min(|3\beta-1|, 1/2)$.
Furthermore,   we prove well-posedness for the system \eqref{wtbsq}-\eqref{wtbsq-data} with surface tension $\beta>1/3$ over a time scale of order $ ( \beta_\ast^{ 1/4} \mu^{ 1/4} h_0 \epsilon^{-1} )^{4/3}$ in one dimension and of order $(   \beta_\ast^{ 1/4} \mu^{ 1/4} h_0 \epsilon^{-1})^{2-}$ in two dimensions, assuming a non-cavitation condition with a parameter $h_0>0$. These results explicitly capture how the existence time depends on all the parameters involved in the model.
Our result for the one dimensional system 
recovers Paulsen's work \cite{P2022} in the long wave regime $\mu\sim \epsilon$ and $\beta_\ast\sim 1$,  whereas it improves it in the weakly nonlinear regime $\epsilon \ll \mu$.
Our two dimensional result appears to be new,  though in the case of $\epsilon=\mu=1$,  local well-posedness was obtained by  Kalisch and Pilod \cite{KP2019}.

Our proofs rely on dispersive and Strichartz estimates combined with compactness arguments.  In what follows,  we sketch the main argument for the Whitham equation \eqref{wt}.  A similar approach extends to the Whitham–Boussinesq system \eqref{wtbsq} (see Sections 4).
Upon differentiating the energy functional associated with \eqref{wt}, we arrive at
    \begin{align*}
       \frac{d}{d t} \norm{\eta (t)}^2_{H^s}
        &\le c \epsilon  \norm{ \partial_{x}  \eta(t) }_{L^\infty_x} \norm{\eta (t)}^2_{H^s},
    \end{align*}
   where  H\"{o}lder's inequality and Kato-Ponce commutator estimate is used to obtain the norms on the right.  Gr\"onwall's Lemma then yields
\begin{align}\label{AT1}
\|\eta\|_{L^{\infty}_TH^s_x} 
 &\le  \exp\left[  \epsilon     \norm{\partial_{x}  \eta }_{L^1_TL^\infty_x} \right]   \norm{ \eta_0}_{H^s} .
    \end{align}
    
In order to control $\norm{\partial_x\eta}_{L^1_TL^\infty_x}$ in \eqref{AT1},  first we write the Duhamel formulation of \eqref{wt},  and then estimate the homogeneous part of solution in the $L^4_TL^\infty_x$ Strichartz norm. The Duhamel's term is estimated using a $L^1_x$--$L^\infty_x$ time-decay property of the propagator; this approach was used by Elgindi and Widmayer  \cite{EW15} in the context of the SQG and inviscid Boussinesq system. These estimates will yield
\begin{align}\label{l1inf}
 \epsilon \norm{\partial_{x}  \eta }_{L^1_TL^\infty_x} \le c \left(  1+ \epsilon (\beta_\ast \mu)^{-\frac 14}T^{\frac 34 } \|\eta\|_{L^{\infty}_TH^s_x}  \right)^2
    \end{align}
which then can be combined with \eqref{AT1} to obtain
\begin{align}\label{AT2}
\|\eta\|_{L^{\infty}_TH^s_x}  \le \exp \left[c \left(  1+ \epsilon  (\beta_\ast \mu)^{-\frac 14}T^{\frac 34 } \|\eta\|_{L^{\infty}_TH^s_x}  \right)^2 \right]  \norm{ \eta_0}_{H^s}  .
    \end{align}
    \vspace{-7mm}

We may now apply a standard bootstrap argument to deduce an a priori bound $\|\eta\|_{L^{\infty}_TH^s_x} \le 4  \norm{ \eta_0}_{H^s}  $ over a time interval of length $$T=C  ( \beta_\ast^{ 1/4} \mu^{ 1/4}  \epsilon^{-1} )^{4/3},$$ where $C$ depends on the initial data norm $\norm{ \eta_0}_{H^s} $.  Together with a standard compactness argument implemented on a regularized
version of the equation, this bound yields the existence of a solution $ \eta \in C([0,T];H^s(\R))$ for  \eqref{wt}-\eqref{wt-data}. Furthermore, uniqueness of solution follows from an
energy estimate applied on the
difference of two solutions, and  continuity of the flow map is obtained via the Bona–Smith method \cite{BS1975}.

In contrast,  the authors in \cite{PSTT25} used a refined Strichartz estimate to obtain \eqref{AT2} 
with $ \mu^{(- 1/5) +}T^{( 4/5) + }$ instead of $ \mu^{ -1/4 }T^{3/4 }$  on the right hand side,  which in turn implies
$$
  \|\eta\|_{L^{\infty}_TH^s_x} \le \exp \left[c \left(  1+ \epsilon  \mu^{-\frac 15 +}T^{\frac 45 + } \|\eta\|_{L^{\infty}_TH^s_x}  \right)^2 \right]  \norm{ \eta_0}_{H^s} . 
$$
    \vspace{-7mm}
    
This estimate yields an a priori bound on $\|\eta\|_{L^{\infty}_TH^s_x} $ over the timescale $\mu^{(1/4) -}\epsilon^{(-5/4) +}$,  from which  well-posedness of \eqref{wt}-\eqref{wt-data} is deduced.

Throughout the paper, we set
$$ \beta_\ast:= \min\left\{|3\beta-1|, \  \frac 12 \right\}. $$

Our first result is as follows. 
\vspace{2mm}
\begin{thm}\label{thm-wt}
    Let $\epsilon,\mu \in (0,1]$, $\beta=0$ or  $\beta>1/3 $. Suppose $\eta_0\in H^s(\R) $ with $s>s_\beta$, where $s_\beta=11/4$
    if $\beta=0$ and $s_\beta=9/4$ if $\beta>1/3$.
Then there exists a time
\begin{equation}\label{wtime}
       T:=T(\epsilon,\mu, \beta)=   \left(c \cdot ( \beta_\ast\mu)^{ -\frac 14} \epsilon \norm{\eta_0}_{H^s} \right)^{ -\frac 43}   , 
    \end{equation}
such that
 \eqref{wt}-\eqref{wt-data} admits a unique solution 
    \begin{equation}\label{eta-solnclass}
            \eta \in C([0,T];H^s(\R)) \cap C^1([0, T];H^{s-1}(\R))
    \end{equation}
  satisfying 
    \begin{equation}\label{eta-solnbound}
    \sup_{0\le t \le  T } \norm{\eta(t)}_{H^s} \le 4  \norm{\eta_0}_{H^s} ,
    \end{equation}
  where $c=c(s) $ is some positive constant.
 Furthermore, the flow map depends continuously on the initial data.
        
\end{thm}

In our second result, we show that the system \eqref{wtbsq}-\eqref{wtbsq-data} is well-posed, provided the initial surface elevation $\eta_0(x)$
 satisfies a non-cavitation condition. 
 This is a physically natural assumption, ensuring that the initial free surface stays strictly above the bottom of the fluid. It reads as follows.
\begin{definition}
Let $d=1$ or $2$ and let $s>\frac{d}{2}$. We say that the initial
elevation $\eta_0\in H^s(\mathbb{R}^d)$ satisfies the
\emph{non-cavitation condition} if
\begin{equation}\label{non-cavt}
\exists  \ h_0\in(0,1)\ \text{such that}\ 
1+\eta_0(x)\ge h_0,\qquad \forall \ x\in\mathbb{R}^d.
\end{equation}
\end{definition}

Throughout, we set $\mathbbm u:=(\eta, \mathbf v)$ with initial data $ \mathbbm u(0)=\mathbbm u_0:= (\eta_0,   \mathbf v_0) $,  and define the space $X_{\mu}^s (\R^d)$ via the norm
\begin{equation*}
\| \mathbbm u\|^2_{X_{\mu}^s} :=\norm{\eta}_{H^s}^2+\norm{\mathbf v}_{H^s}^2+  \norm{  \sqrt{ \mathcal K_{\mu, \beta}} \mathbf v}_{H^s}^2
\end{equation*}
In the case $1/3<\beta\le \beta_0$ for some constant $\beta_0 > 0$, we have
$$ \sqrt{\mathcal K_{\mu,\beta}(\xi) } \sim_{\beta_0} \left\langle \sqrt{ \mu }\, \xi \right\rangle^\frac12 \sim 1+ \mu^\frac12 |\xi|^\frac12. $$
Consequently, 
\begin{equation*}
\| \mathbbm u\|^2_{X_{\mu}^s} \sim \norm{\eta}_{H^s}^2+\norm{\mathbf v}_{H^s}^2+ \sqrt \mu  \norm{ |D|^\frac12   \mathbf v}_{H^s}^2.
\end{equation*}

 We also adopt the shorthand 
$$p_0=1 + \epsilon \norm{\mathbbm u_0}_{L^\infty_x}.$$

Our second result is as follows. 
\begin{thm}\label{thm-wtbq}
     Let $\epsilon,\mu \in (0,1]$,    $\beta>1/3 $,   $d\in \{1,2\}$ and $s>d/4 +5/2$. Suppose that $ \mathbbm u_0:= (\eta_0,   \mathbf v_0) \in X^s_\mu(\R^d) $, where $\eta_0$ 
    satisfy the non-cavitation condition \eqref{non-cavt}  and that $\mathrm{curl}\:\mathbf{v}_0=0$ when $d=2$. Then the Cauchy problem
     \eqref{wtbsq}-\eqref{wtbsq-data} admits a unique solution 
    \begin{equation}\label{U-solnclass}
         \mathbbm u :=(\eta, \mathbf v) \in C([0,T_d];X_\mu^s(\R^d)) \cap C^1([0,T_d];X_\mu^{s-1}(\R^d))
    \end{equation}
    with existence time
    \begin{equation}\label{btime}
       T_d:= T_d(\epsilon,\mu, \beta)=    \left( c h_0^{-1}  (\beta_\ast \mu)^{- \frac {1-\delta_d}4} \epsilon  p_0 \norm{ \mathbbm u_0}_{X^s_\mu} \right)^{ -\frac {4}{4-d+2 { \delta_d }     }},
    \end{equation}
 where $c=c(d,s)$, $ \delta_1=0$ and $ \delta_2>0$ is sufficiently small. Moreover, the solution satisfies the bound
    \begin{equation}\label{U-solnbound}
      \sup_{0 \le t \le T_d} \norm{ \mathbbm u(t)}_{X_\mu^s} \le 4\sqrt{ \frac {p_0 } {h_0}}    \norm{ \mathbbm u_0}_{X_\mu^s} .
    \end{equation}
   Furthermore, the flow map depends continuously on the initial data.

\end{thm}

\begin{remark}
   The proofs of the aforementioned theorems exploit the dispersive properties of the propagators $S_{m_\beta}(\pm t)$ associated to the equations \eqref{wtbsq} and \eqref{wt}.  In the range, $0<\beta \le 1/3$, however, there is no lower bound estimate for the second derivative of the phase function, 
   $|m''_\beta (r)|$ (see Section \ref{sec5}, Lemma \ref{lm:fbeta-threshold}), causing the standard stationary phase analysis—and consequently the underlying dispersive estimates to break down. 
   The same degeneracy appears directly in the asymptotic expansion of the nonlocal Fourier multiplier  $\mathcal K_{\mu, \beta}$ 
  given in \eqref{K-exp}.

\end{remark}

The paper is organized as follows. In Section 2, we introduce the notation, recall commutator and bilnear estimates, and establish the dispersive and Strichartz estimates associated with equations \eqref{wtbsq} and \eqref{wt}. In Section 3, we derive energy and a priori estimates for solutions to \eqref{wt}, from which the result in  Theorem \ref{thm-wt} follows.  In Section 4,  we establish corresponding energy and a priori estimates for solutions to \eqref{wtbsq}, which yield the result stated in Theorem \ref{thm-wtbq}. In the last section, we present lower bound estimates for the first and second order derivatives of the phase function associated to the equations \eqref{wtbsq} and \eqref{wt}.

\section{Notation and preliminary estimates}

\subsection{Notation}
For any positive numbers $a$ and $b$, the notation $a\lesssim b$ stands for $a\le cb$, where $c$ is a positive constant that may change from line to line. Moreover, we denote $a \sim b$  when  $a \lesssim b$ and $b \lesssim a$.  If $A$ and $B$ are two operators, then $[A,B]$ denotes the commutator between
    $A$ and $B$, i.e.
    \[
        [A,B]f = ABf - BAf.
    \]
We denote the Riesz transform by $\mathcal{R}:=D/|D|$, where $D=-i\nabla$ and $|D|=\sqrt{-\Delta}$. For any $s\in\mathbb{R}$, $J^s$ will denote the Bessel potential of order $-s$, defined by 
 $J^s f=\langle D \rangle^s f$,
    where  $\langle r \rangle:= (1+|r|^2)^\frac12$.
    In particular,  the $L^2$-based Sobolev space $H^s$ and its homogeneous counter part $\dot H^s$ can be defined via the norms
    $\|f\|_{H^s} = \|J^s f\|_{L^2}$ and $\|f\|_{\dot H^s} = \||D|^s f\|_{L^2}$.

 Let
$
\rho(s)
=\chi\left(s\right)-\chi \left(2s\right),
 $ where
  $\chi$ is a smooth cutoff function with the following properties:
\begin{equation*}
\chi \in C_0^{\infty}(\mathbb R), \quad 0 \le \chi \le 1, \quad
\chi_{|_{[-1,1]}}=1 \quad \mbox{and} \quad  \mbox{supp}(\chi)
\subset [-2,2].
\end{equation*}
 We then define the frequency projection $P_\lambda$ by
\begin{align*}
\widehat{P_{\lambda} f}(\xi)  = \rho_\lambda(|\xi|)\widehat { f}(\xi) 
 \end{align*}
 for 
  $\lambda \in  2^\Z$, where $\rho_{\lambda}(s):=\rho\left(s/\lambda\right)$.
 Thus, $\supp \rho_\lambda= 
\{ s\in \R: \lambda/ 2 \le |s| \le 2\lambda \}$. 
 Sometimes, we write $f_\lambda: =P_\lambda f$.

  For $s\in\R$, $1\le p\le\infty$ and $1\le q<\infty$,
 the homogeneous and inhomogeneous Besov spaces $\dot B^s_{p,q}$ and $ B^s_{p,q}$ are  defined via the norms 
\[
\|f\|_{\dot B^s_{p,q}}
\sim
\left(
\sum_{\lambda\in 2^{\mathbb Z}}
\lambda^{sq}
\|P_\lambda f\|_{L^p}^{q}
\right)^{1/q} \quad \text{and} \quad  \|f\|_{B^s_{p,q}}
\sim
\left(
\sum_{\lambda\in 2^{\mathbb Z}}
\langle\lambda\rangle^{sq}
\|P_\lambda f\|_{L^p}^{q}
\right)^{1/q}.
\]
In the special case $p=q=2$, we obtain the homogeneous and inhomogeneous Sobolev spaces:
$  \dot B^s_{2,2} = \dot H^s$ and $ B^s_{2,2} = H^s$.
If $T>0$, we define the spaces $L^q\big((0,T) ; X\big)$ and $L^q\big( \mathbb R ; X\big)$ respectively through the norms
$$
\|f\|_{L^q_TX} = \left( \int_0^T \|f(\cdot,t)\|_{X}^q dt \right)^{\frac1q} \quad  \textrm{and} \quad \|f\|_{L^q_tX} = \left( \int_{\mathbb R} \|f(\cdot,t)\|_{X}^q dt \right)^{\frac1q} \, ,
$$
when $1 \le q < \infty$, with the usual modifications when $q=+\infty$.

 \subsection{Commutator, Bilinear and Dispersive estimates}

\begin{lemma}[Kato-Ponce estimates \cite{Kato1988}]
    Let $s \geq 0$, $p\in (1,\infty)$, $f\in H^s(\R^d)$, and $g\in H^{s-1}(\R^d)$.  Then
    \begin{align}\label{kato_ponce}
        \norm{[J^s,f] g}_{L^p_x} \lesssim \norm{\nabla f}_{L^\infty_x} \norm{J^{s-1} g}_{L^p_x} + \norm{J^s f}_{L^p_x} \norm{ g}_{L^\infty_x}.
    \end{align}
    \vspace{2mm}
    Moreover,  let $p_1,p_4 \in (1,\infty]$ and $p_2,p_3\in (1,\infty)$ be such that $\frac{1}{p}= \frac{1}{p_1} + \frac{1}{p_2} = \frac{1}{p_3} + \frac{1}{p_4}$. Then
    \begin{align}\label{eq:kato_ponce_ineq}
        \norm{J^s(f g)}_{L^p_x} \lesssim \norm{f}_{L^{p_1}_x} \norm{J^{s} g}_{L^{p_2}_x}  + \norm{J^s f}_{L^{p_3}_x} \norm{ g}_{L^{p_4}_x}
    \end{align}
    for $f,g\in H^s(\R^d)$.
\end{lemma}

\begin{definition}[Symbol class {\cite[Definition B.7]{L13}}] \label{def:symb_class}
    We say that a symbol $m(D)$ is an element of the symbol class $\mathbb S^s$ with $s\in \R$, if $\xi \in \R^d \mapsto m (\xi) \in \mathbb{C}$ is smooth and satisfies 
    \begin{align*}
        \forall \alpha \in \N^d, \ \ \sup_{\xi \in \R^d} \angles{\xi}^{|\alpha|- s} \Big| \frac{\partial^\alpha}{\partial\xi^\alpha} m(\xi)  \Big| < \infty. 
    \end{align*}
    We also introduce the following semi-norm
    \begin{align*}
        \mathcal{N}^s(m) = \sup_{\alpha\in \N^d, \:  |\alpha| \leq 2+d+\lceil\frac{d}{2} \rceil}\: \sup_{\xi \in \R^d} \angles{\xi}^{|\alpha|-s} \Big| \frac{\partial^\alpha}{\partial \xi^\alpha }  m(\xi) \Big|.
    \end{align*}
\end{definition}
This symbol class satisfies the following commutator estimate found in \cite[Lemma B.8]{L13} (see also \cite{L2006}).

\begin{lem}\label{lm:gen_kato_ponce}
    Let $s\geq 0$, $s_0>\frac{d}{2}$, and $m \in \mathbb S^s$. If $f\in H^s(\R^d)\cap H^{s_0+1}(\R^d)$ and $g\in H^{s-1}(\R^d) \cap H^{t_0}(\R^d)$, then
    \begin{align}\label{eq:gen_kato_ponce_Linfty}
        \norm{[m(D),f]g}_{L^2_x} &\lesssim \mathcal{N}^s(m) (\norm{\nabla f}_{L^\infty_x}\norm{g}_{H^{s-1}_x} + \norm{\nabla f}_{H^{s-1}_x}\norm{g}_{L^\infty_x}).
\end{align}
\end{lem}
\vspace{2mm}
\begin{lemma} \label{lm-kpn}
Let $\mu \in (0,1]$, $ \beta > 1/3$, $s \ge 0$, $s_0 > \frac{d}{2}$,
$f \in H^s(\mathbb{R}^d)\cap H^{s_0+1}(\mathbb{R}^d)$,
and
$g \in H^{s-1}(\mathbb{R}^d)\cap H^{s_0}(\mathbb{R}^d)$.
Then
\begin{equation}
\label{commuest_K}
\left\|
\bigl[ \mathcal  K_{\mu,\beta}^\frac12 J^s ,f\bigr]g
\right\|_{L_x^2}
\lesssim
\|\nabla f\|_{L_x^\infty}
\left\|  J^\frac12_{\mu} J^{s-1} g\right\|_{L_x^2}
+
\left\| J^\frac12_{\mu} J^sf\right\|_{L_x^2}
\|g\|_{L_x^\infty},
\end{equation}
where $J_{\mu}=  \angles{\sqrt \mu D}   $.
\end{lemma}

\begin{proof}
We may assume $1/3<\beta\le \beta_0$ for some constant $\beta_0 > 0$.
Observe that
$$ \mathcal K_{\mu,\beta}^\frac12(\xi)=\left\langle \sqrt{\beta \mu }\, \xi \right\rangle \cdot   \mathcal T_{\mu}^\frac12(\xi),$$
and hence
$$\mathcal K_{\mu,\beta}^\frac12(\xi) \sim \left\langle \sqrt{\beta \mu }\, \xi \right\rangle \angles{\sqrt \mu \xi}^{-1/2} \lesssim_{\beta_0} \angles{\sqrt \mu \xi}^{1/2} .$$
Moreover, the derivatives of $\left\langle \sqrt{\beta \mu }\, \xi \right\rangle$ and $\sqrt{ \mathcal T_{\mu} (\xi)}$ satisfy the estimates 
\begin{align*}
      \left| \partial_\xi^\alpha \left\langle \sqrt{\beta \mu }\, \xi \right\rangle\right|
 & \lesssim_\alpha (\beta \mu)^{|\alpha|/2} \left\langle \sqrt{\beta \mu}\, \xi \right\rangle^{1-|\alpha|}
\end{align*}
and \begin{align*}
  \left| \partial_\xi^\alpha \sqrt{ \mathcal T_{\mu} (\xi)}\right |
 \lesssim_\alpha  \mu^{|\alpha|/2} \left\langle \sqrt{ \mu}\, \xi \right\rangle^{-\frac 12-|\alpha|}.
\end{align*}
for all $k\ge 0$.
Applying the Leibniz rule, we obtain
\begin{equation}\label{Kdervest}
\begin{split}
          \left| \partial_\xi^\alpha  \mathcal K_{\mu,\beta}^\frac12(\xi)\right|
  &\lesssim_\alpha (1+\beta)^{|\alpha|/2}\,\mu^{|\alpha|/2}\,
\angles{\sqrt{\beta\mu}\,\xi}\,\angles{\sqrt{\mu}\,\xi}^{-\frac12-|\alpha|}
\\
 &\lesssim_{\alpha, \beta_0} \mu^{|\alpha|/2}\,
\angles{\sqrt{\mu}\,\xi}^{\frac12-|\alpha|}.
\end{split}
\end{equation}
Define 
$m(\xi):= \mathcal K_{\mu,\beta}^\frac12(\xi) \angles{\xi}^s .$
We split this function into its low and high frequency components, $m= m_{ l}+ m_{ h}$, where
\[
 m_{l}(\xi):=\chi(\sqrt\mu\,\xi) m (\xi),\qquad
m_{h}(\xi):=\big(1-\chi(\sqrt\mu\,\xi)\big)m(\xi).
\]
Any derivative falling on the $\chi$ produces $\mu^{1/2}$ on a set where $\sqrt\mu|\xi|\sim1$ (the derivatives of $\chi$ are supported on the annulus $|\xi|\sim \mu^{-1/2}$). On this set, we have $ \angles{\xi} \ge |\xi|  \sim \mu^{-1/2}$ and $ \angles{\xi} \le \mu^{-1/2}  \angles{\sqrt\mu\xi} \sim  \mu^{-1/2}$. Thus,   $ \mu^{1/2}\sim\angles{\xi}^{-1}$. 
Together with \eqref{Kdervest} this gives, for both pieces,
\[
  |\partial^\alpha_\xi   m_\bullet |\lesssim\angles{\xi}^{s-|\alpha|}\,\angles{\sqrt\mu\,\xi}^{\frac12}\,\mathbf 1_{\operatorname{supp} m_\bullet}.
\]
On the support of $m_{l}$, we have $\angles{\sqrt\mu\,\xi}\lesssim1$, and so $m_{l}\in\mathbb S^{s}$, whereas on the support of $m_{h}$, we have $ \angles{\sqrt\mu\,\xi} \sim \sqrt\mu \angles{\,\xi}$, and so $m_{h}\in\mathbb S^{s+\frac12}$. Moreover, 
    $$ \mathcal N^{s}(m_{l})\lesssim 1, \qquad \mathcal N^{s+\frac12 }(m_{h})\lesssim \mu^\frac14.$$ 
Then applying Lemma \ref{lm:gen_kato_ponce} to each piece, $m_{l}$ or $m_{h}$, yields
\begin{equation}\label{mlh}
    \begin{split}
  \norm{[m_{l}(D),f]g}_{L^2}&\lesssim\norm{\nabla f}_{L^\infty}\norm{J^{s-1}g}_{L^2}+\norm{J^sf}_{L^2}\norm{g}_{L^\infty},\\
  \norm{[m_{h}(D),f]g}_{L^2}&\lesssim\mu^{\frac14}\Big(\norm{\nabla f}_{L^\infty}\norm{J^{s-\frac12}g}_{L^2}+\norm{J^{s+\frac12}f}_{L^2}\norm{g}_{L^\infty}\Big).
\end{split}
\end{equation}
Next, using the fact that $J_{\mu} (\xi)\ge1$ and  $\mu^{1/4}\angles{\xi}^{1/2}\le\angles{\sqrt\mu\,\xi}^{1/2}=J_{\mu}^{1/2}$, and then applying Plancherel, 
we obtain
\begin{equation}\label{JJmu}
    \begin{split}
    \norm{J^{s-1}g}_{L^2}\le\norm{J_{\mu}^{1/2} J^{s-1}g}_{L^2},
    \\\quad \mu^{1/4}\norm{J^{s-\frac12}g}_{L^2}\lesssim\norm{J_\mu^{1/2} J ^{s-1}g}_{L^2}, 
    \\
    \quad \mu^{1/4}\norm{J^{s+\frac12}f}_{L^2}\lesssim\norm{ J_\mu^{1/2}J^{s} f}_{L^2}.
\end{split}
\end{equation}
Finally, since $\mathcal  K_{\mu,\beta}^\frac12 J^s=m_l + m_h $, we can combine \eqref{mlh}
 with \eqref{JJmu} to obtain the 
desired estimate \eqref{commuest_K}.

\end{proof}

   \begin{lemma} If $0<s<r$,  then
   
\begin{equation}\label{BH}
\|fg\|_{\dot B^{\,s}_{1,1}}
\lesssim
\|f\|_{H^{r}} \|g\|_{H^{r}}
\end{equation}
for all $f, g\in H^r(\R^d)$.

\end{lemma}
\begin{proof}
Observe that
\[
\|J^{-r}f\|_{\dot B^{\,s}_{1,1}} \lesssim
\sum_{\lambda}
\lambda^{s} \angles{\lambda}^{-r}
\|f_\lambda\|_{L^1}
\lesssim
\left(
\sum_{\lambda}
\frac{\lambda^s}{\angles{\lambda}^{r}}
\right)
\|f\|_{L^1} \lesssim \|f\|_{L^1}
\]
since the dyadic sum in parentheses is finite, as we see by splitting into
$0<\lambda<1$ and $\lambda\ge1$.
Then using \eqref{eq:kato_ponce_ineq} with $p_j=2$, we obtain
\begin{align*}
\|fg\|_{\dot B^s_{1,1}} \lesssim  \| J^r(fg)\|_{L^1} &\lesssim  \left(  \| J^r f \|_{L^2} \| g \|_{L^2}+ \| f \|_{L^2}\| J^r g \|_{L^2} \right) \\
&\lesssim  \|f\|_{H^r} \|g\|_{H^r}.
\end{align*}
\end{proof}

We also need the following Bernstein inequality, which is valid for $1\le p \le r \le \infty$
(see, for instance, \cite[Appendix A]{Tao}):
\begin{equation}\label{Bern-est1}
 \|  f_{\lambda} \|_{L^r_x}   \lesssim \lambda^{\frac dp -\frac dr} \|f_{\lambda}\|_{L^p_x},
\end{equation}
Moreover, we have for all $s_1, s_2 \in \R$ and $p\ge 1$,
\begin{equation}\label{Bern-est2}
 \| |D|^{s_1} \mathcal K_{\mu, \beta}^{s_2} \mathcal R f_{\lambda} \|_{L^p_x}   \lesssim  \lambda^{s_1} \angles{ \sqrt{\beta \mu}  \lambda}^{s_2}  \|f_{\lambda}\|_{L^p_x}.
\end{equation}

Next, we derive frequency-localized dispersive estimates for linear Whitham-type equation 
\begin{equation}\label{wt-h}
 \begin{split}
     i u_t -  m_{\beta, \mu, d}(D) u =0\\
     u(x,0) = f(x),
 \end{split} 
\end{equation}
for $(x,t) \in \R^d \times \R$, where the Fourier multiplier $m_{\beta, \mu, d}(D)$ has symbols 
\begin{align*}
 m_{\beta, \mu, 1}(\xi)&=\xi \angles{\sqrt{\beta \mu} \xi} \sqrt{\mathcal T _\mu(\xi)}
   \qquad  \qquad   (\xi \in \R),
\\
m_{\beta, \mu,2}(\xi) &=|\xi|  \angles{\sqrt{\beta \mu} \xi} \sqrt{\mathcal T _\mu(\xi)}  \qquad  \qquad   (\xi \in \R^2).
\end{align*}

Now, consider the phase function 
$$m_\beta(r)=r \angles{  \sqrt{ \beta} r}   \sqrt{\mathcal T (r)} $$ where $\mathcal T =\mathcal T_1$.
We have the following estimates for the first and second order derivatives of
$m_\beta$.
\begin{lemma}\label{lm-mderv-est}
Let $\beta =0$ or $\beta>1/3 $. Then for all $ r \ge 0$, we have 
\begin{align}
\label{m-1stderv-est} 
&0<m_{\beta}'(r) \sim \angles{ \sqrt{  \beta} r} \angles{r}^{-1/2}
\\
\label{m-2ndderv-est} 
&\beta_\ast r  \angles{ \sqrt{ \beta }  r }\angles{ r}^{-5/2} \lesssim |m_{\beta} ''(r)| \lesssim   r  \angles{ \sqrt{ \beta }  r }\angles{ r}^{-5/2} ,
\end{align}
for $\beta_\ast:= \min\left\{|3\beta-1|, \   1/2 \right\}, $
 where the hidden constants are independent of $\beta$.

\end{lemma}
The estimates in Lemma \ref{lm-mderv-est} were derived for $\beta \in \{0, 1\}$ in  \cite {PSST21} (see also \cite{DST20, DDT22-2}). However, the lower bound estimate in \eqref{m-2ndderv-est} is new. These estimates combined with stationary phase analysis yield the dispersive estimates in Lemma \ref{lm-dispstrz-est} below..

The solution to \eqref{wt-h}, which is given by, 
\begin{equation}\label{soln-gp}
     u(t)=S_{ \beta,  \mu, d}(t) f = e^{-i t m_{\beta, \mu, d}(D)} f,
\end{equation}
satisfies the following dispersive and Strichartz estimates
for frequency localized function $f$.

\begin{lemma}[Localised dispersive and Strichartz estimates ] \label{lm-dispstrz-est}
Let $\mu \in (0,1]$, $\beta=0$ or $1/3<\beta \le 1$ and $\beta_\ast=3\beta-1$. Then we have the following:

\begin{enumerate}[(a)]
    \item In one dimension, we have the dispersive estimate
\begin{align}
    \label{disp-est-1d}
\| \mathcal S_{ \beta, \mu, 1} (t)  P_\lambda f  \|_{L^\infty_x(\R)}
 \lesssim (\beta_\ast \mu)^{- \frac 12} \lambda^{ -\frac 12} \angles{ \sqrt{\beta \mu }  \lambda}^{-\frac{1}{2}}   \angles{ \sqrt{\mu}  \lambda}^{\frac{5}{4 }}  |t|^{-\frac 12 }  \| f\|_{L_x^1(\R)} .
\end{align}

\vspace{2mm}
Moreover, we have the Strichartz estimate
\begin{align}  \label{str-est-1d}
\norm{ \mathcal S_{\beta, \mu, 1} ( t)  P_\lambda f }_{ L^{4}_{t} L^{\infty }_{ x} (\R^{1+1}) } \lesssim  (\beta_\ast \mu)^{- \frac 14} \lambda^{ -\frac 14} \angles{ \sqrt{\beta \mu }  \lambda}^{-\frac{1}{4}}   \angles{ \sqrt{\mu}  \lambda}^{\frac{5}{8 }}
\norm{  f}_{ L^2_{ x}(\R )} .
\end{align}
\vspace{2mm}
\item In two dimensions, we have the dispersive estimate
\begin{align}
 \label{disp-est-2d}
\| \mathcal S_{ \beta, \mu, 2} (t)  P_\lambda f  \|_{L^\infty_x(\R^2)}
 \lesssim    (\beta_\ast \mu)^{- \frac {1-\theta}2} \lambda^{ 2\theta} \angles{ \sqrt{\beta \mu }  \lambda}^{-1+\theta}   \angles{ \sqrt{\mu}  \lambda}^{\frac{3}{2 }(1-\theta)}  |t|^{-1+\theta }  \| f\|_{L_x^1(\R^2)} 
\end{align}
for $0\le \theta \le 1$.

\vspace{2mm}
Moreover, we have the Strichartz estimate
\begin{align}
 \label{str-est-2d}
\norm{ \mathcal S_{\beta, \mu, 2} ( t) P_\lambda f}_{ L^{\frac 2{1- \delta}}_t L^{\frac 2\delta}_x ( \R^{2+1}) } \lesssim  (\beta_\ast \mu)^{- \frac {1-\delta}4} \angles{ \sqrt{\beta \mu }  \lambda}^{-\frac 12 + \frac \delta 2}   \angles{ \sqrt{\mu}  \lambda}^{\frac{3}{4 }(1-\delta)} 
\norm{  f}_{ L^2_{ x}(\R^2 )} ,
\end{align}
for sufficiently small $\delta > 0$.

\end{enumerate}

\end{lemma}

\begin{proof}
The Strichartz estimates \eqref{str-est-1d} and  \eqref{str-est-2d} follow from the dispersive estimates \eqref{disp-est-1d} and \eqref{disp-est-2d}, respectively, upon application of a classical duality argument (see e.g.,  \cite{DDT22-2} for the argument).
 It suffices to prove the dispersive estimates \eqref{disp-est-1d} and \eqref{disp-est-2d} for $\mu=1$. Indeed, if \eqref{disp-est-1d} holds for $\mu=1$, i.e., 
\begin{align}
    \label{disp-estm1-1d}
\| \mathcal S_{ \beta, 1, 1} (t)  P_\lambda f  \|_{L^\infty_x(\R)}
 \lesssim \beta_\ast ^{- \frac 12} \lambda^{ -\frac 12} \angles{ \sqrt{\beta }  \lambda}^{-\frac{1}{2}}   \angles{   \lambda}^{\frac{5}{4 }}    |t|^{-\frac 12 }  \| f\|_{L_x^1(\R)},
\end{align}
then we can use
the rescaling property,
    \begin{equation}\label{smu-scaling}
        \left(    \mathcal S_{ \beta, \mu, d} (t)  P_\lambda f  \right)(x)= \left(    \mathcal S_{ \beta, 1, d} \left( \mu^{-1/2}t\right) P_{ \sqrt \mu \lambda }f (\sqrt \mu \cdot) \right) \left( \mu^{-1/2}x\right),
        \end{equation}
to obtain
\begin{align*}
\| \mathcal S_{ \beta, \mu, 1} (t)  P_\lambda f  \|_{L^\infty_x(\R)}
& \lesssim    \| \mathcal S_{ \beta, 1, 1} \left( \mu^{-1/2}t\right)  P_{ \sqrt \mu \lambda }f (\sqrt \mu \cdot)  \|_{L^\infty_x(\R)}
\\
& \lesssim   \beta_\ast ^{- \frac 12}  (\sqrt{\mu } \lambda)^{ -\frac 12} \angles{ \sqrt{\beta \mu }  \lambda}^{-\frac{1}{2}}   \angles{  \sqrt{ \mu }  \lambda}^{\frac{5}{4 }}     \left| \mu^{-1/2}t\right|^{-\frac 12}    \|  f(\sqrt \mu \cdot)\|_{L_x^1(\R)} 
\\
& ´=   (\beta_\ast \mu)^{- \frac 12} \lambda^{ -\frac 12} \angles{ \sqrt{\beta \mu }  \lambda}^{-\frac{1}{2}}   \angles{ \sqrt{\mu}  \lambda}^{\frac{5}{4 }}   \left| t\right|^{-\frac 12}    \|  \|  f\|_{L_x^1(\R)} .
 \end{align*}

Similarly, by the rescaling property \eqref{smu-scaling}, estimate \eqref{disp-est-2d} reduces to 
\begin{align}
    \label{disp-estm1-2d}
\| \mathcal S_{ \beta, 1, 2} (t)  P_\lambda f  \|_{L^\infty_x(\R^2)}
 \lesssim \beta_\ast ^{- \frac 12} \lambda^{ 2\theta} \angles{ \sqrt{\beta  }  \lambda}^{-1+\theta}   \angles{  \lambda}^{\frac{3}{2 }(1-\theta)}  |t|^{-1+\theta }  \| f\|_{L_x^1(\R^2)}.
\end{align}

It remains to prove \eqref{disp-estm1-1d} and \eqref{disp-estm1-2d}. These estimates are proved in \cite[see Theorem 1]{DDT22-2} for $\beta \in \{0, 1\}$ (to be precise \eqref{disp-estm1-2d} is proved for $\theta=0$ but 
we can use interpolation\footnote{ The estimate \eqref{disp-estm1-2d} for $0\le \theta \le 1$ can be obtained by interpolating the estimate for $\theta=0$ and the trivial estimate 
\begin{align*}
\| \mathcal S_{ \beta, 1, 2} (t)  P_\lambda f  \|_{L^\infty_x(\R^2)}
 \lesssim \lambda^{ 2}  \| f\|_{L_x^1(\R^2)}.
\end{align*}
}  to obtain the estimate for $0\le \theta \le 1$).
Finally, one can follow the argument in the proof of \cite[see Theorem 1]{DDT22-2}, applying Lemma \ref{lm-mderv-est}, to obtain \eqref{disp-estm1-1d} and \eqref{disp-estm1-2d} for $\beta=0$ or $\beta > 1/3$.

\end{proof}

\begin{remark}
    The estimate in \eqref{disp-est-1d}--\eqref{str-est-2d} still hold if we replace $P_\lambda  $ by $\widetilde P_\lambda=P_{\lambda/2}+P_\lambda+P_{2\lambda}$.
Now, observe that $P_\lambda =\widetilde P_\lambda P_\lambda$. Consequently, by \eqref{disp-est-1d},
\begin{align*}
\| \mathcal S_{ \beta, \mu, 1} (t)  P_\lambda f  \|_{L^\infty_x(\R)} =\| \mathcal S_{ \beta, \mu, 1} (t)   \widetilde P_\lambda  f_\lambda   \|_{L^\infty_x(\R)}
 \lesssim   a_{\beta, \mu}(\lambda)   \left| t\right|^{-\frac 12 }  \|  f_\lambda\|_{L_x^1(\R)} .
 \end{align*}
 which is \eqref{str-est-1d} with $f$ on the right replaced by $f_\lambda$. Similarly, we can replace $f$ on the right hand side of \eqref{str-est-1d}--\eqref{str-est-2d} by $f_\lambda$.
\end{remark}

\section{A prior  estimates for solutions of \eqref{wt} \& Existence theory  }

We define the modified energy associated with the Whitham equation \eqref{wt} by
\begin{align*}
    \mathcal{E}_s(\eta(t)) = \int_\R \Big(J^s \eta \Big)^2 \d x =\norm{\eta (t)}^2_{H^s},
\end{align*}
which is conserved for $s=0$, i.e.,  $ \mathcal{E}_0(\eta(t)) = \mathcal{E}_0(\eta_0) $.

\subsection{Energy estimates}
 Differentiating $ \mathcal{E}_s(\eta(t)) $ with respect to $t$,   and then using \eqref{wt} and integration by parts,  we obtain
    \begin{align*}
         \frac{1}{2}\frac{d}{d t} \mathcal{E}_s (\eta(t))  &= \int_\R J^s \eta J^s \eta_t \, dx \\
         &= - \int_\R  J^s \eta J^s    \mathcal{K}_{\mu, \beta} \eta_x  \, dx  - \epsilon \int_\R  J^s \eta J^s (\eta\eta _x ) \, dx
         \\
         &=- \frac12  \int_\R   \partial_x \left[ \sqrt{ \mathcal{K}_{\mu, \beta}} J^s  \eta    \right]^2  \, dx  - \epsilon \int_\R  J^s \eta [J^s,  \eta] \eta _x  \, dx+  \frac \epsilon 2  \int_\R  \eta_x \left(J^s  \eta    \right)^2  \, dx .
    \end{align*}
The first term on the right is zero, whereas we can use H\"{o}lder's inequality and the Kato-Ponce commutator estimate \eqref{kato_ponce} to estimate the second and third terms, and obtain
    \begin{align*}
       \frac{d}{d t} \norm{\eta (t)}^2_{H^s}
        &\le c \epsilon \norm{\eta_x(t) }_{L^\infty_x} \norm{\eta (t)}^2_{H^s}.
    \end{align*}
 It follows that
     \begin{align*}
       \frac{d}{d t} \norm{\eta (t)}_{H^s}
        &\le c \epsilon \norm{\eta_x(t) }_{L^\infty_x} \norm{\eta (t)}_{H^s}
    \end{align*}
which yields,  applying Gr\"onwall's lemma,
\begin{equation}\label{eta-grnw}
\norm{ \eta(t)}_{H^s}   \le  \exp\left( c   \epsilon   \int_0^t \norm{\eta_x (\tau)}_{L^\infty_x} \, d\tau \right) \norm{ \eta_0}_{H^s} .
    \end{equation}

\subsection{A prior estimate for $ \int_0^T  \norm{\eta_x (t)}_{L^\infty_x} \, dt$}

We can now use the dispersive property of of solution to \eqref{wt} in Lemma \ref{lm-dispstrz-est}
 to obtain the following bound for $ \int_0^T  \norm{\eta_x (t)}_{L^\infty_x} \, dt$.

\begin{lemma}\label{lm-etax}
     Let $\mu, \epsilon \in (0,1]$,  $\beta=0$ or  $\beta >1/3 $,  and  $s>s_\beta$ with $s_\beta=11/4$
    if $\beta=0$ and $s_\beta=9/4$ if $\beta>1/3$.
      For a solution $\eta \in C([0,T]; H^{s}(\R))$ of \eqref{wt}--\eqref{wt-data}, we have the a prior estimate
   \begin{align}
       \label{etax-est}
    \int_0^T  \norm{\eta_x (t)}_{L^\infty_x} \, dt  &\lesssim  (\beta_\ast \mu)^{- \frac 1 4}  T^{ \frac34}  \| \eta_0\|_{ H^{s }}+ \epsilon (\beta_\ast \mu)^{- \frac 1 2}   T^{ \frac 32}    \| \eta \|^2_{ L^\infty_T H^{s }  }.
    \end{align}
    
    \end{lemma}
\begin{proof}  Equation \eqref{wt} can be written as 
$$i \eta_t -  m_{\mu, \beta, 1}(D) \eta =-i\epsilon\eta \eta_x, $$ and so the Duhamel formulation for
\eqref{wt}--\eqref{wt-data} becomes
\begin{align*}
    \eta(t) & = S_{\mu, \beta,  1}( t) \eta_0  - \epsilon \int_0^t  S_{\mu, \beta,  1}(t-\tau)  (\eta \eta_x)(\tau)   \, d\tau
    \\
    &:=  \eta^{\text{hom}}(t) + \eta^{\text{inh}}(t).
\end{align*}
We have 
\begin{align*}
\int_0^T  \norm{\eta_x (t)}_{L^\infty_x} \, dt &\lesssim   \sum_\lambda  \int_0^T \norm{P_\lambda \partial_ x \eta^{\text{hom}}(t)}_{L^\infty_x} \, dt + \sum_\lambda  \int_0^T \norm{P_\lambda \partial_ x \eta^{\text{inh}} (t)}_{L^\infty_x} \, dt 
\\
&\lesssim  \underbrace{T^\frac34  \sum_\lambda  \norm{ P_\lambda \eta^{\text{hom}}  }_{ L^4_TL^\infty_x (\R) } }_{:=A_T}+ \underbrace{ \sum_\lambda  \int_0^T \norm{P_\lambda \partial_ x \eta^{\text{inh}} (t)}_{L^\infty_x} \, dt }_{:=B_T}, 
\end{align*}
where we used H\"{o}lder in time to get the first term on the second line. We consider only the case $\beta >1/3$, as the case $\beta =0$, is similar.
Assume $s\ge 9/4 +\delta$ for some $\delta>0$. 
Applying \eqref{str-est-1d}, we estimate the first term on the second line as
\begin{align*}
A_T & \lesssim T^\frac34  (\beta_\ast \mu)^{- \frac 1 4}   \sum_\lambda \lambda^{\frac 34 }  \angles{ \sqrt{ \beta \mu }  \lambda}^{ -\frac {1}4 }  \angles{ \sqrt{\mu}  \lambda}^{ \frac 58 }\norm{    P_\lambda  \eta_0 }_{L^2_x (\R) } 
\\
&\lesssim   T^\frac34(\beta_\ast \mu)^{- \frac 1 4}   \sum_\lambda \angles{  \lambda}^{ \frac {9}8 }\norm{    P_\lambda  \eta_0 }_{L^2_x (\R) } 
\\
&\lesssim   T^\frac34 (\beta_\ast \mu)^{- \frac 1 4}  \norm{     \eta_0 }_{H^{ s} }
.\end{align*}
On the other hand,  we use \eqref{disp-est-1d},  to estimate the second term as
\begin{align*}
B_T& \lesssim  \epsilon  \int_0^T  \int_0^t  \sum_\lambda   \norm{  P_\lambda   \partial_x S_{\mu, \beta, 1} (t-\tau) (\eta \eta_x)  (\tau) }_{ L^\infty_x (\R) } d\tau dt
\\
&\lesssim \ \epsilon  (\beta_\ast \mu)^{- \frac 1 2}    \int_0^T  \int_0^t (t- \tau )^{-\frac12} \underbrace{ \sum_\lambda \lambda^{\frac 32} \angles{ \sqrt{\mu \beta }  \lambda}^{ -\frac {1}2 }  \angles{ \sqrt{\mu}  \lambda}^{ \frac 54 } \norm{    P_\lambda  (\eta^2 ) (\tau) }_{L^1_x (\R) }}_{:= C(\tau)}  \, d\tau dt
\end{align*}
Using the bilinear estimate  \eqref{BH},  we obtain
\begin{align*}
C(\tau) \lesssim  \sum_\lambda (\lambda^{\frac 32 }+ \lambda^{\frac {9}4 }) \norm{    P_\lambda  (\eta^2 ) (\tau) }_{L^1_x (\R) } 
&\lesssim \norm{  
\eta^2(\tau) }_{ \dot B^\frac {3} 2_{1,1} } +  \norm{  
\eta^2(\tau) }_{ \dot B^{\frac {9}4}_{1,1} } 
\\
&\lesssim  \norm{  
\eta(\tau) }^2_{ H^{ s}} . 
\end{align*}

Combining the above estimates, we obtain
\begin{align*}
\int_0^T  \norm{\eta_x (t)}_{L^\infty_x} \, dt & \lesssim T^\frac34  (\beta_\ast \mu)^{- \frac 1 4}   \norm{     \eta_0 }_{H^{s} }+  \epsilon (\beta_\ast \mu)^{- \frac 1 2}   \int_0^T  \int_0^t (t- \tau )^{-\frac12} \norm{  
\eta(\tau) }^2_{ H^{s}}  d\tau dt
\\
& \lesssim T^\frac34  (\beta_\ast \mu)^{- \frac 1 4}  \norm{     \eta_0 }_{H^{s} } +\epsilon (\beta_\ast \mu)^{- \frac 1 2}    T^{ \frac 32}    \| \eta \|^2_{ L^\infty_TH^{s}  }.
\end{align*}
 This concludes the proof of Lemma \ref{lm-etax}.
\end{proof}

\subsection{Control of the solution norm $\norm{ \eta}_{L_T^\infty H^s}$  }
We will now combine \eqref{eta-grnw} and Lemma \ref{lm-etax} with a bootstrap argument to have a control of the solution norm $\norm{ \eta}_{L_T^\infty H^s}$.

\begin{lemma}
\label{lm-wt}
    Let $\mu, \epsilon \in (0,1]$,  $\beta=0$ or  $1/3 <\beta\le 1$ ,  and  $s>s_\beta$ with $s_\beta=11/4$
    if $\beta=0$ and $s_\beta=9/4$ if $\beta>1/3$.
   Assume  
$       \eta\in C([0,T^\ast);H^s(\R) ) 
$    is a solution to  \eqref{wt}-\eqref{wt-data}
    with $\eta_0\in H^s(\R)  $,  where the solution is defined on its maximal time of existence and satisfying
the blow-up alternative:
\begin{align}\label{BA}
    \text{If } \ \ T^\ast <\infty, \text{ then } \limsup_{t\nearrow T^\ast} \norm{ \eta(t)}_{H^s} = \infty.
\end{align}

Then there exists a positive time 
    \begin{equation}\label{T}
       T=  \left(c (\beta_\ast \mu)^{- \frac 1 4} \epsilon \norm{\eta_0}_{H^s} \right)^{ -\frac 43}    
    \end{equation}
    such that $T\ast > T$ and  
    \begin{align*}
        \sup_{t\in [0,T]} \norm{ \eta(t)}_{H^s} \le 4 \norm{ \eta_0}_{H^s},
    \end{align*}   
    where $c=c(s)$ is a positive constant.
\end{lemma}

\begin{proof}
    Define
    \begin{align}\label{Ttil}
        \Tilde{T} := \sup\left\{ T \in (0,T^\ast) :    \sup_{t\in [0,T]} \norm{ \eta(t)}_{H^s} \le  4\norm{\eta_0 }_{ H^s} \right\}.
    \end{align}
We will argue by contradiction by assuming that $\Tilde{T} < T$.
    Combining \eqref{eta-grnw}, \eqref{etax-est}, 
\eqref{Ttil} and \eqref{T},  we obtain for $0\le t \le \min(T, \tilde T)$,
\begin{align*}
  \norm{ \eta(t)}_{H^s}   & \le  e^{ c_0   \epsilon   \int_0^T  \norm{\eta_x (t)}_{L^\infty_x} \, dt }   \norm{ \eta_0}_{H^s} 
  \\
  & \le   e^{ c_1      \left [ \epsilon (\beta_\ast \mu)^{- \frac 1 4}  T^{ \frac34}  \| \eta_0\|_{ H^{s }}+ \epsilon^2 (\beta_\ast \mu)^{- \frac 1 2}   T^{ \frac 32}    \| \eta_0 \|^2_{  H^{s }  } \right]}   \norm{ \eta_0}_{H^s} 
  \\
  & \le 2   \norm{ \eta_0}_{H^s} 
\end{align*}
by choosing the constant $c$ in \eqref{T} larger that  $ 2c_1/ \ln 2$.
By continuity, we deduce that there exists some $\tau \in  (\widetilde T, T^\ast)$ such that
      $\norm{ \eta(\tau)}_{H^s} \le 4  \norm{ \eta_0}_{H^s} $.   This contradicts the
definition of $\widetilde T$. Therefore,  $ T \le \widetilde T $, where $ T$ is as in \eqref{T}. This concludes the proof of  Lemma \ref{lm-wt}.
\end{proof}

\subsection{Proof of Theorem \ref{thm-wt}  }

\subsubsection*{Existence of solution and Continuity
of the flow map}
With the a priori estimate in Lemma  \ref{lm-wt} at hand, the complete proof of existence of solutions would
then follow from a standard compactness argument implemented on a regularized
version of the equation (for details, see for instance \cite[Section 4.1]{PSST21}).
Furthermore, continuity of the flow map is obtained via the Bona–Smith method \cite{BS1975}.  As this argument is well known, we omit the details and refer the interested reader to \cite{LPS2018}.

\subsubsection*{Uniqueness  of solution  }
This is a consequence of energy estimates for the
diﬀerence of two solutions.  
Suppose that $\eta_ 1 $ and $\eta_2$ be two solutions of the equation 
\eqref{wt} in the class \eqref{eta-solnclass} with initial data $\eta_ 1 (0)=f_1$, $\eta_2(0)=f_2$.  By Sobolev embedding and \eqref{eta-solnbound},  we have for $s>3/2$,
\begin{equation} \label{etaj-solnbound}
  \|  \partial_x \eta_j(t)\|_{ L_x^\infty}\lesssim  \|  \eta_j(t)\|_{ H^s} \lesssim  \|  f_j\|_{ H^s}   \qquad \forall  \ t \in [0,  T]
\end{equation}
where $T=T(\epsilon,\mu, \beta)$ is as in \eqref{wtime}.

Set
$
\zeta := \eta_1 - \eta_2$. 
Then $  \zeta$ solves 
\begin{equation} \label{wt-dif}
\partial_t \zeta + \mathcal K_{\mu, \beta}(D) \partial_x \zeta = -\epsilon (\eta_1 \partial_x \zeta + \zeta \partial_x \eta_2)
\end{equation}
with initial data $\zeta (0)=\zeta_0:=f_1-f_2$.

We want to estimate $\zeta $ in $L^2(\R)$.  Multiplying \eqref{wt-dif} by $\zeta$ and integrating by parts, we deduce that
\begin{align*}
\frac{1}{2} \frac{d}{dt} \int_{\mathbb{R}} \zeta^2 \, dx&= - \epsilon\int_{\mathbb{R}} \zeta  \mathcal K_{\mu, \beta}  \partial_x \zeta  \, dx  - \epsilon  \int_{\mathbb{R}} ( \zeta \eta_1 \partial_x \zeta + \zeta^2 \partial_x \eta_2) \, dx
\\
&=  \frac  \epsilon 2  \int_{\mathbb{R}}   \partial_x \eta_1 \zeta^2  \, dx  - \epsilon  \int_{\mathbb{R}}   \partial_x \eta_2 \zeta^2 \, dx.
\end{align*}
By H\"{o}lder's inequality and \eqref{etaj-solnbound}, 
\begin{align*}
\frac{d}{dt} \norm{ \zeta(t)}^2_{L^2_x}   &\lesssim  \epsilon\left(  \|  \partial_x \eta_1(t)\|_{ L_x^\infty} + \   \|  \partial_x\eta_2(t)\|_{L_x^\infty}  \right)
  \norm{ \zeta(t)}^2_{L^2_x}  
  \\
  &\lesssim \epsilon \left(  \|  f_1\|_{ H^s}+  \|  f_2\|_{ H^s}\right)
  \norm{ \zeta(t)}^2_{L^2_x} .
\end{align*}
 Gr\"onwall's inequality implies
\begin{equation*}
 \sup_{0\le t \le  T } \norm{ \zeta(t)}_{L^2_x}  \lesssim \norm{ \zeta_0}_{L^2_x} ,
    \end{equation*}
  which provides the uniqueness result in Theorem \ref{thm-wt}, by choosing $f_1=f_2$.

\section{A prior  estimates for solutions of \eqref{wtbsq} \& Existence theory  }

As a consequence of the curl-free condition on the initial datum
in Theorem \ref{thm-wtbq}, it follows that \text{curl}  $\mathbf v=0$.  Indeed, applying the curl on the second equation in 
\eqref{wtbsq}, and then
using the fundamental theorem of calculus and \text{curl} $ \mathbf v_0=0$,  we obtain \text{curl}  $\mathbf v=0$.  In other words,
we have the following relation 
\begin{equation}\label{curl}
 \partial_{x_i} v_j =\partial_{x_j} v_i, \qquad (i, j=1, 2).
\end{equation}

Let  $ \mathbbm u:=(\eta, \mathbf v)$ with initial data $ \mathbbm u(0)= \mathbbm u_0:= (\eta_0,   \mathbf v_0) $.  Define the modified energy associated with the Whitham-Boussinesq system  \eqref{wtbsq} by
\begin{align*}
 \mathcal E_s( \mathbbm u(t))&=\frac{1}{2}\int_{\R^d} (J^s\eta)^2+
| \sqrt{\mathcal K_{\mu,\beta }} J^s \mathbf v|^2 + \epsilon  \eta |J^s \mathbf v|^2\, dx,
\end{align*}
which is conserved for $s=0$, i.e.,  $ \mathcal{E}_0( \mathbbm u(t)) = \mathcal{E}_0( \mathbbm u_0) $ (see  the proof of Lemma \ref{lm-energy1d} below).

\subsection{Energy estimates}

\begin{lemma}\label{lm-energy1d}
Let $\mu, \epsilon \in (0,1]$,  $\beta>1/3$,    $d\in \{1,2\}$,  $s> d/2+1$. Suppose that  
$U=(\eta,  \mathbf v) \in C([0,T]; X_ \mu ^s(\mathbb{R}^d))$ is a solution 
to \eqref{wtbsq}
on a time interval $[0,T]$ for some $T>0$. 
Assume that there exists $ h_0 \in (0,1)$ such that
\begin{equation}\label{noncav}
 1+ \varepsilon \eta(x,t) \ge   h_0 \quad \forall \ (x,t)\in \mathbb{R}^d \times [0,T].
\end{equation}
Then, the functional $ \mathcal E_s(U(t))$ satisfies the coercivity estimate
\begin{equation}\label{EX-eq}
h_0 \,\| \mathbbm u(t)\|_{X_\mu^s}^2 
\le 2  \mathcal E_s( \mathbbm u(t)) 
\le   p_d(t) \,\| \mathbbm u(t)\|_{X_\mu^s}^2.
\end{equation}
and the energy estimate
\begin{equation}\label{E-derv}
\frac{ d}{dt} \mathcal E_s( \mathbbm u(t)) \le c_0     \epsilon  h_0^{-1}  p_d(t) q_d(t) \mathcal E_s( \mathbbm u(t))
\end{equation}
for some positive $c_0=c_0(s)$,
where 
\begin{align}\label{pd}
p_d(t):&= 1+ \epsilon \left(\norm{ \eta}_{L^\infty_x(\R^d) } +\norm{\mathbf v}_{L^\infty_x (\R^d)}  \right)
\\
\label{qd}
q_d(t):&= \norm{ \nabla \eta (t)}_{L^\infty_x(\R^d)} +\norm{\nabla \mathbf v (t)}_{L^\infty_x (\R^d)}  +  \norm{  \mathcal K_{\mu,\beta } \nabla \mathbf v (t)}_{L^\infty_x (\R^d)} .
\end{align}

\end{lemma}

\begin{proof}We use the shorthand $\mathcal K:=\mathcal K_{\mu, \beta}$.
Clearly, from the non-cavitation condition, we obtain
 \begin{align*}
2 \mathcal E_s( \mathbbm u(t)) &= \int_{\R^d} (J^s\eta)^2+
| \sqrt{\mathcal K} J^s \mathbf v|^2 + (1+\epsilon  \eta) |J^s \mathbf v|^2-|J^s \mathbf v|^2\, dx
 \\
 & \ge  h_0 \| \mathbbm u(t)\|_{X_\mu^s}^2 
\end{align*}
and 
by H\"{o}lder’s inequality, 
\begin{align*}
2 \mathcal E_s( \mathbbm u(t)) &\le    \left[ \norm{\eta}_{H^s}^2+    \norm{ \sqrt{ \mathcal K}   \mathbf   v}_{H^s}^2  +\epsilon \norm{\eta}_{L^\infty_x} \norm{\mathbf  v}_{H^s}^2  \right] 
 \\
 &\le    p_d(t) \| \mathbbm u(t)\|_{X_\mu^s}^2 .
\end{align*}
This proves \eqref{EX-eq}.

Next, we prove \eqref{E-derv}.  In view of \eqref{EX-eq}, it suffices to show
\begin{equation}\label{E-derv1}
\frac{ d}{dt} \mathcal E_s( \mathbbm u(t)) \le c_0  \epsilon p_d(t)q_d(t)   \| \mathbbm u(t)\|_{X_\mu^s}^2.
\end{equation}
We give the proof of \eqref{E-derv1} only in two dimensions ($d=2$) as the argument for the one dimensional case is similar (and even easier).

Differentiating $\mathcal E_s( \mathbbm u(t))$ with respect to $t$,  using the system \eqref{wtbsq} for $d=2$ and and applying integration by parts, we obtain
\begin{align*}
\frac{ d}{dt} \mathcal E_s( \mathbbm u(t)) &=-\int_{\R^2} J^s \eta J^s \mathcal  K \nabla \cdot \mathbf  v \, dx  -\epsilon \int_{\R^2} J^s (\eta \nabla \cdot \mathbf v) J^s \eta  \, dx -\epsilon \int_{\R^2}  J^s ( \nabla \eta \cdot  \mathbf v) J^s \eta   \, dx
\\ &-\int_{\R^2} \sqrt \mathcal  K J^s \nabla \eta \cdot \sqrt \mathcal  K J^s   \mathbf v \, dx 
- \frac \epsilon 2 \int_{\R^2}   \sqrt \mathcal  K J^s \nabla (|\mathbf v |^2) \cdot \sqrt \mathcal  K J^s \mathbf v \, dx
\\ &
- \epsilon  \int_{\R^2}  \eta J^s \mathbf v \cdot  J^s \nabla \eta dx  -\frac{\epsilon^2}2 \int_{\R^2}  \eta J^s \mathbf v \cdot J^s \nabla (|\mathbf v |^2) \,  dx
\\
&-\frac{\epsilon}{2}  \int_{\R^2}  \mathcal  K \nabla \cdot \mathbf v |J^s \mathbf v|^2 \, dx    -\frac{\epsilon^2}{2}   \int_{\R^2} \nabla \cdot (\eta \mathbf v) |J^s \mathbf v|^2\,  dx
\\
&:=\sum_{k=1}^{9} I_k(t).
\end{align*}
We remark that if $s=0$,  we can easily see via integration by parts that 
$I_1+I_4=0$,  $I_2 + I_3+ I_6=0$, $I_5+I_8=0$ and $I_7+I_9=0$.  Therefore,  $\sum_{k=1}^{9} I_k(t)=0$, and hence we get the conservation of energy $\mathcal E_0(U(t)) =\mathcal E_0(U_0) $.

Next, we consider the case $s\neq 0$.
By integration by parts we see that $I_1=-I_4$,  so $I_1+I_4=0$.  Moreover,  we can use integration by parts to write 
\begin{align*}
I_2(t) &= -\epsilon  \int_{\R^2} [J^s , \eta]\nabla \cdot \mathbf  v J^s \eta \, dx- \epsilon \int_{\R^2} \eta   J^s \nabla \cdot \mathbf  v J^s \eta  \, dx
\\
I_3(t)&= -\epsilon
\int_{\R^2} [J^s , \mathbf v] \cdot \nabla \eta J^s \eta \, dx+ \frac12 \epsilon \int_{\R^2} \nabla \cdot \mathbf  v (J^s \eta)^2 \, dx 
\\
I_6(t)&= \epsilon
\int_{\R^2}  \nabla \eta  \cdot   J^s \mathbf v J^s \eta \, dx+ \epsilon
\int_{\R^2}  \eta J^s \nabla \cdot \mathbf v J^s \eta \, dx.
\end{align*}

Note that the second terms in $I_2$ and $I_6$ cancel each other out. Hence By H\"{older},  \eqref{kato_ponce} and Sobolev inequalities, we obtain
\begin{align*}
\bigabs{ I_2(t) + I_3(t)+I_6(t) }
& \lesssim    \epsilon   \left[ \norm{ \nabla \mathbf v}_{L^\infty_x}   \norm{J^s  \eta}_{L^2_x} +  \norm{J^s \mathbf v}_{L^2_x} \norm{ \nabla \eta}_{L^\infty_x} \right]  \norm{J^s  \eta}_{L^2_x} 
\\
& \qquad +\epsilon  \norm{ \nabla \eta }_{L^\infty_x}    \norm{J^s  \mathbf v}_{L^2_x}   \norm{J^s  \eta}_{L^2_x} +   \epsilon \norm{ \nabla\mathbf v }_{L^\infty_x}      \norm{J^s  \eta}^2_{L^2_x} 
  \\
  & \lesssim   \epsilon \left[ \norm{ \nabla \mathbf v}_{L^\infty_x} +  \norm{ \nabla \eta}_{L^\infty_x}  \right]     \| \mathbbm u\|^2_{X_{\mu}^s}
  \\
&\lesssim  \epsilon q_2(t) \| \mathbbm u(t)\|_{X_\mu^s}^2.
\end{align*}
Next, we use the curl free condition \eqref{curl},  and integration by parts to write  $I_7$ as
\begin{align*}
I_{7}(t)&= -\epsilon^2  \sum_{i, j=1}^2 
\int_{\R^2}  \eta J^s \mathbf v_i   J^s ( \mathbf v_j \partial_{x_i} \mathbf v_j )  \, dx 
\\
&= -\epsilon^2  \sum_{i, j=1}^2 
\int_{\R^2}  \eta J^s \mathbf v_i   J^s ( \mathbf v_j \partial_{x_j} \mathbf v_i )  \, dx 
\\
&= -\epsilon^2  \sum_{i, j=1}^2 
\int_{\R^2}  \eta J^s \mathbf v_i  [ J^s , \mathbf v_j] \partial_{x_j} \mathbf v_i   \, dx - \frac {\epsilon^2} 2  \sum_{i, j=1}^2 
\int_{\R^2}  \eta v_j  \partial_{x_j} ( J^s \mathbf v_i)^2    \, dx
\\
&= -\epsilon^2  \sum_{i=1}^2 
\int_{\R^2}  [ J^s , \mathbf v] \cdot \nabla \mathbf v_i \eta J^s \mathbf v_i     \, dx + \frac {\epsilon^2} 2  \int_{\R^2} \nabla \cdot (\eta \mathbf v) |J^s \mathbf v|^2\,  dx.
\end{align*}
The second terms in $I_7$ cancel out $I_9$ . Hence by H\"{older},  \eqref{kato_ponce}  and Sobolev inequalities, we obtain
\begin{align*}
\bigabs{ I_7(t) + I_9(t)}
& \lesssim    \epsilon^2   \norm{ \nabla \mathbf v}_{L^\infty_x}   \norm{J^s  \mathbf v}_{L^2_x}   \norm{ \eta}_{L^\infty_x}\norm{J^s   \mathbf v }_{L^2_x} 
  \\
  & \lesssim   \epsilon  (  \epsilon \norm{ \eta}_{L^\infty_x})\left[ \norm{ \nabla \mathbf v}_{L^\infty_x} +  \norm{ \nabla \eta}_{L^\infty_x}  \right]     \| \mathbbm u\|^2_{X_{\mu}^s}
  \\
&\lesssim  \epsilon p_2(t) q_2(t) \| \mathbbm u(t)\|_{X_\mu^s}^2.
\end{align*}

Similarly, 
we use the curl free condition on $ \mathbf v$ and apply integration by parts to write  $I_5$ as
\begin{align*}
I_{5}(t)
&= -\epsilon  \sum_{i=1}^2 
\int_{\R^2}  [ \sqrt \mathcal  K J^s , \mathbf v] \cdot \nabla \mathbf v_i \sqrt \mathcal  K J^s \mathbf v_i     \, dx + \frac {\epsilon} 2  \int_{\R^2} \nabla \cdot  \mathbf v | \sqrt \mathcal  K J^s \mathbf v|^2\,  dx.
\end{align*}
So, by H\"{older} and Lemma \ref{lm-kpn},
\begin{align*}
|I_5(t)|& \le \epsilon   \norm{ \nabla \mathbf v}_{L^\infty_x} \left(  \norm{  \sqrt \mathcal  J_\mu  J^s \mathbf v}_{L^2_x}+ \norm{  \sqrt \mathcal  K  J^s \mathbf v}_{L^2_x} \right) \norm{  \sqrt \mathcal  K  J^s \mathbf v}_{L^2_x} \lesssim  \epsilon  q_2(t) \| \mathbbm u(t)\|_{X_\mu^s}^2,
\end{align*}
where to obtain the second inequality we used $\mathcal  J_\mu\sim_{\beta_0} \mathcal  K_{\mu, \beta}$, which holds for $\beta \in (1/3, \beta_0)$. 
Finally, 
\begin{align*}
|I_8(t)|& \le \epsilon   \norm{ \mathcal  K \nabla \mathbf v }_{L^\infty_x}  \norm{ J^s \mathbf v}^2_{L^2_x}  \lesssim  \epsilon  q_2(t) \| \mathbbm u(t)\|_{X_\mu^s}^2.
\end{align*}
Combining the above estimates yield \eqref{E-derv1} for $d=2$.

\end{proof}

\subsection{A prior estimate for $\int_0^T  q_d(t) \, dt $}

We  use the dispersive properties of the system \eqref{wtbsq} in Lemma \ref{lm-dispstrz-est}
 to obtain the following bound for $\int_0^T  q_d(t) \, dt $, where $q_d$ is given in \eqref{qd}.

\begin{lemma}\label{lm-qdest}
     Let $d\in \{1,2\}$,  $s>d/4 +5/2$, $\mu, \epsilon \in (0,1]$, $\beta >1/3$, and $T>0$.  For a solution $ \mathbbm u:=(\eta, \mathbf v) \in C([0,T]; X_\mu^{s}(\R^d))$ of \eqref{wtbsq}--\eqref{wtbsq-data}.
    Then
\begin{align}
       \label{qdest}
     \int_0^T  q_d(t) \, dt &\lesssim    \sigma_d(\beta, \mu, T) \| \mathbbm u_0\|_{ X_\mu^{s }}+ \epsilon \ \sigma^2_d(\beta, \mu, T) \| \mathbbm u  \|^2_{ L^\infty_T X_\mu^{s }  },
    \end{align}
where
\begin{equation}\label{sigmaT}
     \sigma_d(\beta, \mu, T)= (\beta_\ast \mu)^{- \frac {1-\delta_d}4} T^{\frac {4-d +2\delta_d}4}
 \end{equation}  
for $ \delta_1=0$ and $ \delta_2=\delta>0$ sufficiently small.
    \end{lemma}
\begin{proof}

We diagonalize
the system \eqref{wtbsq} in one dimension ($d=1$), by defining
\begin{align*}
u_\pm:=\frac{\eta \mp \sqrt { \mathcal K_{\mu,\beta }}  v}{2  },  
\end{align*}
and hence
\begin{equation}\label{upm-1d}
\eta=u_++ u_-, \qquad  v= -  \mathcal K_{\mu,\beta } ^{-\frac12 }   (u_+ - u_-).
\end{equation}

Consequently, the system \eqref{wtbsq} in one dimension transforms to 
\begin{equation}
\label{wtbsq-1d}
 (i\partial_t\pm m_{\mu, \beta,  1}(D) ) u_\pm
=  \frac \epsilon 2  \underbrace{ \left[ -i( \eta v)_x  \pm i     \sqrt { \mathcal K_{\mu,\beta }} ( vv_x )\right]}_{:=F_\pm  (\eta, v)},
\end{equation}
and the corresponding initial data become
\begin{equation}
\label{wtbsq-1d-data}
u_\pm (0):= f_\pm = \frac{\eta_0 \mp  \sqrt { \mathcal K_{\mu,\beta }} v_0}{2  }\in H^s(\R).
\end{equation}
The Duhamel formulation for
\eqref{wtbsq-1d}--\eqref{wtbsq-1d-data}  
becomes
\begin{equation}\label{upm-duhm}
\begin{split}
    u_\pm(t) &= S_{\mu, \beta, 1}(\mp t) f_\pm  - \epsilon \int_0^t  S_{\mu,  \beta, 1}(\mp(t-\tau) )  F_\pm  (\eta, v)(\tau)   \, d\tau
    \\
     &:= u^{\text{hom}}_\pm(t) + u^{\text{inh}}_\pm(t).
\end{split}
\end{equation}

To diagonalize
the system \eqref{wtbsq} in two dimension ($d=2$), we define
\begin{align*}
w_\pm=\frac{\eta \mp \sqrt { \mathcal K_{\mu,\beta }} \mathcal R \cdot  \mathbf  v}{2  }.
\end{align*}

Since $ \mathbf v$ is a curl-free vector field, i.e.,
$\nabla \times \mathbf v=0$,  we have
$$
\nabla \nabla \cdot \mathbf  v = \Delta \mathbf  v=-|D|^2
\quad \Rightarrow \quad   \mathbf v= \mathcal R ( \mathcal R \cdot   \mathbf v)$$
which in turn implies
\begin{equation}\label{wpm-2d}
\eta=w_++ w_-, \qquad  \mathbf v= -  \mathcal K_{\mu,\beta } ^{-\frac12 }  \mathcal R (w_+ - w_-).
\end{equation}
Note also that for curl-free $\mathbf  v$, we have $\nabla ( |\mathbf v|^2)= 2(\mathbf v \cdot \nabla )\mathbf v$.

Consequently,  the system \eqref{wtbsq} in two dimension transforms to 
\begin{equation}\label{wtbsq-2d}
 (i\partial_t\pm m_{\mu, \beta,  2}(D) ) w_\pm
=   \frac \epsilon 2  \underbrace{ \left[-i \nabla \cdot (\eta \mathbf{v})  \pm i    \sqrt { \mathcal K_{\mu,\beta }} \mathcal R \cdot ( ( \mathbf v  \cdot \nabla)   \mathbf v ) \right]}_{:=G_\pm  (\eta,  \mathbf v)}  
\end{equation}
and the corresponding initial data becomes
\begin{equation}
\label{wtbsq-2d-data}
w_\pm (0):= g_\pm = \frac{\eta_0 \mp \sqrt { \mathcal K_{\mu,\beta }} \mathcal R \cdot  \mathbf  v_0}{2  }\in H^s(\R^2).
\end{equation}

The Duhamel formulation for
the Cauchy problem \eqref{wtbsq-2d}--\eqref{wtbsq-2d-data} 
becomes
\begin{equation}\label{wpm-duhm}
\begin{split}
w_\pm(t)& = S_{\mu, \beta,  2}(\mp t) g_\pm  -  \epsilon \int_0^t  S_{\mu, \beta, 2}(\mp(t-\tau) )  G_\pm  (\eta,  \mathbf v)(\tau)   \, d\tau
\\
     &:= w^{\text{hom}}_\pm(t) + w^{\text{inh}}_\pm(t).
\end{split}
\end{equation}

\begin{remark}
    
The propagators $S_{\mu, \beta,  d}(\pm t)$ in \eqref{upm-duhm}  and  \eqref{wpm-duhm} satisfy the estimates in Lemma \ref{lm-dispstrz-est}.
\end{remark}

\subsubsection*{ \underline{Proof of \eqref{qdest} for $d=1$}}

Recall that $$ q_1(t)= \norm{ \eta_x (t)}_{L^\infty_x(\R)} +\norm{ v_x (t)}_{L^\infty_x (\R)}  +  \norm{  \mathcal K_{\mu,\beta } v_x (t)}_{L^\infty_x (\R)}. $$
From \eqref{upm-1d} and  \eqref{upm-duhm},  we get
 \begin{align*}
\int_0^T  q_1(t) \, dt  &\le  \sum_\pm \int_0^T \sum_{j=-1}^1    \norm{  K_{\mu,\beta } ^{\frac j2}  \partial_x u_\pm (t)}_{L^\infty_x (\R)} \, dt
\\
&\lesssim   \sum_\pm \int_0^T \underbrace{\sum_\lambda \sum_{j=-1}^1     \norm{ P_\lambda K_{\mu,\beta } ^{\frac j2}  \partial_x  u^{\text{hom}}_\pm(t)}_{L^\infty_x (\R)}}_{:= q^\pm_{11}(t)} \, dt+  \sum_\pm  \int_0^T \underbrace{\sum_\lambda \sum_{j=-1}^1     \norm{ P_\lambda K_{\mu,\beta } ^{\frac j2}  \partial_x  u^{\text{inh}}_\pm(t)}_{L^\infty_x (\R)}}_{:= q^\pm_{12}(t)} \, dt.
 \end{align*}
 
Fix $s= 11/4+\delta$ for any $\delta>0$.  Applying H\"{o}lder in time and then \eqref{str-est-1d},   we get
    \begin{align*}
 \int_0^T  q_{11}^\pm(t) \, dt 
& \lesssim T^\frac34 \sum_\lambda \sum_{j=-1}^1   
   \norm{    P_\lambda  K_{\mu,\beta } ^{\frac j2}  \partial_x S_{\mu, \beta, 1}(\pm t) f_\pm }_{ L^4_TL^\infty_x (\R) }
   \\
   & \lesssim T^\frac34
  (\beta_\ast \mu)^{- \frac 14} \sum_\lambda \lambda^{\frac34
 }  \left(\sum_{j=-1}^1      \angles{ \sqrt{\mu}  \lambda}^{ \frac38 +\frac j2 }\right) \norm{    P_\lambda   f_\pm }_{ L^2_x (\R) } 
 \\& \lesssim T^\frac34
   (\beta_\ast \mu)^{- \frac 14} \sum_\lambda 
 \angles{\lambda}^{\frac {13}8} \norm{    P_\lambda   f_\pm }_{ L^2_x (\R) } 
 \\
& \lesssim T^\frac34
   (\beta_\ast \mu)^{- \frac 14}  \norm{       f_\pm }_{ H^{\frac {13} 8+\delta  }   }
 \\
 &\lesssim   T^\frac34 (\beta_\ast \mu)^{- \frac 14} \|  \mathbbm u_0\|_{X_{\mu}^{s }  }, 
 \end{align*}
 where to get the last line we used 
 $$
 \norm{       f_\pm }_{ H^r}
 \lesssim  \|\eta_0 \|_{ H^{r}}+  \| \sqrt \mathcal K_{\mu,\beta }  v_0  \|_{ H^{r}}
 \lesssim   \|  \mathbbm u_0\|_{X_{\mu}^{r}}.
 $$

On the other hand,  we use \eqref{disp-est-1d} to obtain
 \begin{align*}
 \int_0^T  q_{12}^\pm(t) \, dt 
& \lesssim  \epsilon (\beta_\ast \mu)^{- \frac 1 2}   \int_0^T  \int_0^t (t- \tau )^{-\frac12} A(\tau)\, d\tau dt
 \end{align*}
 where
\begin{align*}
 A(\tau) 
=\sum_\lambda \sum_{j=-1}^1  \lambda^{\frac12 }          \angles{ \sqrt{\mu}  \lambda}^{ \frac34 +\frac j2 } \left[ \lambda \norm{  P_\lambda 
(\eta v)(\tau) }_{ L^1_x (\R)} + \angles{ \sqrt{\mu}  \lambda}^\frac12 \norm{  P_\lambda 
(v v_x)(\tau) }_{ L^1_x (\R)  } \right] .
 \end{align*}
This can be estimated as
\begin{align*}
  A(\tau) & \lesssim  
  \sum_\lambda  \left[  (\lambda^{\frac 32
 }  +\lambda^\frac{11}4)  \norm{  P_\lambda 
(\eta v)(\tau) }_{ L^1_x (\R)} +  (\lambda^{\frac 12
 }  +\mu^\frac 78\lambda^\frac {9}4)\norm{  P_\lambda 
(v v_x)(\tau) }_{ L^1_x (\R)  } \right]
\\
& \lesssim  
     \norm{  
(\eta v)(\tau) }_{ \dot B^\frac {3} 2_{1,1} } + \norm{  
(\eta v)(\tau) }_{ \dot B^\frac {11} 4_{1,1} }+ \norm{  
(v v_x)(\tau) }_{ { \dot B^\frac 1 2_{1,1} }  } + \mu^\frac 14 \norm{  
(v v_x)(\tau) }_{  \dot B^ \frac94_{1,1}   }  
\\
& \lesssim    \|\eta(\tau) \|_{H^{\frac {11} 4+\delta}}  \| v(\tau)\|_{H^{\frac {11} 4+\delta}}  +   \| v(\tau)\|^2_{H^{\frac32+\delta}} + \mu^\frac 14 \| v(\tau) \|_{H^{\frac 94+\delta}}   \norm{  
v_x(\tau) }_{H^{\frac {9} 4+\delta}} 
\\
& \lesssim   \|  \mathbbm u(\tau) \|^2_{X_{\mu}^ {s}},
 \end{align*}
where to get the third line we used \eqref{BH} and to get the last line we used 
 $$
 \mu^\frac 14   \norm{  
v_x(\tau) }_{H^{\frac {9} 4+\delta}}\lesssim  \mu^\frac 14  \| |D|^\frac12 v (\tau) \|_{H^{\frac{11}4 +\delta}} 
 \\
   \lesssim   \|  \mathbbm u(\tau) \|^2_{X_{\mu}^ {s}} .
 $$

 Combining the above estimates, we obtain
\begin{align*}
\int_0^T  q_1(t) \, dt &\lesssim   (\beta_\ast \mu)^{- \frac 14} T^\frac34 \|  \mathbbm u_0\|_{X_{\mu}^{ \frac {13} 8+\delta}}+ \epsilon (\beta_\ast \mu)^{- \frac 1 2} \int_0^T  \int_0^t (t- \tau )^{-\frac12}  \|  \mathbbm u(\tau)\|^2_{X_{\mu}^{\frac {11} 4+\delta}}d\tau dt
\\
&\lesssim  (\beta_\ast \mu)^{- \frac 14} T^\frac34 \|  \mathbbm u_0\|_{X_{\mu}^{s}}+ \epsilon (\beta_\ast \mu)^{- \frac 1 2} T^{ \frac 32}     \|  \mathbbm u \|^2_{ L^\infty_T X_\mu^{s  }  }
 \end{align*}
 which is the desired estimate \eqref{qdest} for $d=1$.

\subsubsection*{ \underline{Proof of \eqref{qdest} for $d=2$}} 
Recall that
 $$
q_2(t)= \norm{ \nabla \eta (t)}_{L^\infty_x(\R^2)} +\norm{\nabla \mathbf v (t)}_{L^\infty_x (\R^2)}  +  \norm{  \mathcal K_{\mu,\beta } \nabla \mathbf v (t)}_{L^\infty_x (\R^2)} .
$$
 From \eqref{wpm-2d} and  \eqref{wpm-duhm},  we get
 \begin{align*}
\int_0^T  q_2(t) \, dt  
&\lesssim   \sum_\pm \int_0^T \underbrace{\sum_\lambda \sum_{j=-1}^1     \norm{ P_\lambda K_{\mu,\beta } ^{\frac j2}  \nabla w^{\text{hom}}_\pm(t) }_{L^\infty_x (\R^2)}}_{:= q^\pm_{21}(t)} \, dt + \sum_\pm \int_0^T \underbrace{\sum_\lambda \sum_{j=-1}^1     \norm{ P_\lambda K_{\mu,\beta } ^{\frac j2}  \nabla w^{\text{inh}}_\pm(t) }_{L^\infty_x (\R^2)}}_{:= q^\pm_{22}(t)} \, dt  .
 \end{align*}

Fix $s=3+ 3\delta$ for any $\delta>0$.  Applying H\"{o}lder in time,  Bernstein inequality \eqref{Bern-est1} and \eqref{str-est-2d},  we get
    \begin{align*}
 \int_0^T  q_{21}^\pm(t) \, dt 
& \lesssim T^{ \frac {1 + \delta}2 } \sum_\lambda \sum_{j=-1}^1   \lambda^{\delta
 }
   \norm{    P_\lambda  K_{\mu,\beta } ^{\frac j2}  \nabla  S_{\mu, \beta, 2}(\pm t) g_\pm }_{ L^{\frac 2{1- \delta}}_T L^{\frac 2 \delta}_x (\R^2) }
   \\
   & \lesssim T^{ \frac {1 + \delta}2 }  (\beta_\ast \mu)^{- \frac {1-\delta} 4}  \sum_\lambda \lambda^{1+\delta
 }  \left(\sum_{j=-1}^1     \angles{ \sqrt{\mu}  \lambda}^{ \frac 14 +\frac j2 }\right) \norm{    P_\lambda   g_\pm }_{ L^2_x (\R^2) } 
 \\& \lesssim T^{ \frac {1 + \delta}2 }  (\beta_\ast \mu)^{- \frac {1-\delta} 4} \sum_\lambda    \angles{  \lambda}^{ \frac 74 +\delta} \norm{    P_\lambda   g_\pm }_{ L^2_x (\R^2) } 
 \\
& \lesssim 
  T^{ \frac {1 + \delta}2 } (\beta_\ast \mu)^{- \frac {1-\delta} 4}  \norm{       g_\pm }_{ H^{ \frac 74 + 2\delta}}
  \\
& \lesssim  T^{ \frac {1 + \delta}2 }  (\beta_\ast \mu)^{- \frac {1-\delta} 4}  \|  \mathbbm u_0\|_{X_{\mu}^{s}}, 
 \end{align*}
where to get the last line we used 
 $$
 \norm{       g_\pm }_{ H^r}
 \lesssim  \|\eta_0 \|_{ H^{r}}+  \| \sqrt \mathcal K_{\mu,\beta }  \mathcal R \cdot  \mathbf v_0  \|_{ H^{r}}
 \lesssim   \|  \mathbbm u_0\|_{X_{\mu}^{r}}.
 $$

On the other hand,  we use \eqref{disp-est-2d} with $\theta=\delta$, to obtain
 \begin{align*}
 \int_0^T  q_{22}^\pm(t) \, dt 
& \lesssim  \epsilon (\beta_\ast \mu)^{- \frac {1-\delta} 2} \int_0^T  \int_0^t (t- \tau )^{-1+\delta} B(\tau)\, d\tau dt
 \end{align*}
 where
\begin{align*}
 B(\tau) 
=\sum_\lambda \sum_{j=-1}^1  \lambda^{1+2\delta }          \angles{ \sqrt{\mu}  \lambda}^{ \frac 12 +\frac j2 } \left[ \lambda \norm{  P_\lambda 
(\eta \mathbf v)(\tau) }_{ L^1_x (\R^2)} + \angles{ \sqrt{\mu}  \lambda}^\frac12 \norm{  P_\lambda 
( \mathbf v  \cdot \nabla)   \mathbf v )(\tau) }_{ L^1_x (\R^2)  } \right].
 \end{align*}
This can be estimated as
\begin{align*}
  B(\tau) & \lesssim  
  \sum_\lambda  \left[  (\lambda^{2+ 2\delta
 }  + \lambda^{3+ 2\delta
 }) \norm{  P_\lambda 
(\eta \mathbf v)(\tau) }_{ L^1_x (\R^2)} +   (\lambda^{1+ 2\delta
 }  + \mu^\frac 14 \lambda^{\frac 52+ 2\delta
 })   \norm{  P_\lambda 
( \mathbf v  \cdot \nabla)   \mathbf v )(\tau) }_{ L^1_x (\R^2)  }    \right]
\\
& \lesssim   \sum_{k=1}^2 \norm{  
(\eta  \mathbf v)(\tau) }_{ \dot B^{1+k+2\delta}_{1,1} } + \norm{  
 P_\lambda 
( \mathbf v  \cdot \nabla)   \mathbf v )(\tau) }_{ { \dot B^{1+2\delta}_{1,1} }  } + \mu^\frac 14  \norm{   P_\lambda 
( \mathbf v  \cdot \nabla)   \mathbf v )(\tau)}_{  \dot B^ {\frac 52 +2\delta}_{1,1}   } 
\\
& \lesssim    \|\eta(\tau) \|_{H^{3+ 3\delta}}  \|  \mathbf v(\tau)\|_{H^{3+ 3\delta}}  +   \| \mathbf v(\tau)\|^2_{H^{2+2 \delta}} + \mu^\frac 14 \| \mathbf v(\tau) \|_{H^{\frac 52+ 3\delta}}   \norm{  \nabla
\mathbf v(\tau) }_{H^{\frac {5} 2+ 3\delta}} 
\\& 
   \lesssim   \|  \mathbbm u(\tau) \|^2_{X_{\mu}^ {s}},
 \end{align*}
where to get the third line we used \eqref{BH} and to get the last line we used 
$$
 \mu^\frac 14   \norm{  
 \nabla
\mathbf v(\tau)  }_{H^{\frac {5} 2+ 3\delta}}\lesssim  \mu^\frac 14  \| |D|^\frac12 \mathbf v (\tau) \|_{H^{3 +3\delta}} 
 \\
   \lesssim   \|  \mathbbm u(\tau) \|^2_{X_{\mu}^ {s}}. 
 $$
 
  Combining the above estimates, we obtain
\begin{align*}
\int_0^T  q_2(t) \, dt &\lesssim  (\beta_\ast \mu)^{- \frac {1-\delta} 4} T^{ \frac {1 + \delta}2 } \|  \mathbbm u_0\|_{X_{\mu}^s} + \epsilon \ (\beta_\ast \mu)^{- \frac {1-\delta} 2}  \int_0^T  \int_0^t(t- \tau )^{-1+\delta}  \|  \mathbbm u(\tau)\|^2_{X_{\mu}^{s}}d\tau dt
\\
&\lesssim   (\beta_\ast \mu)^{- \frac {1-\delta} 4}  T^{ \frac {1 + \delta}2 } \|  \mathbbm u_0\|_{X_{\mu}^{s}}  +\epsilon  (\beta_\ast \mu)^{- \frac {1-\delta} 2} T^{ 1+\delta}     \|  \mathbbm u \|^2_{ L^\infty_T X_\mu^{ s } }.
 \end{align*}
 which is the desired estimate \eqref{qdest} for $d=2$.

\end{proof}

\subsection{Control of the solution norm $\norm{  \mathbbm u}_{L_T^\infty X^s_\mu}$ }

We will now combine the estimates in Lemma \ref{lm-energy1d} and Lemma \ref{lm-qdest} with a bootstrap argument to have a control of the solution norm $\norm{  \mathbbm u}_{L_T^\infty X^s_\mu}$.

\begin{lemma}
\label{lm-ap-wtbsq}
   Let $s>d/4 +5/2$, $\mu, \epsilon \in (0,1]$ and $\beta >1/3$.  
   Assume  
$        \mathbbm u:= (\eta, v) \in C([0,T^\ast);X^s_\mu(\R^d) ) 
$    is a solution to  \eqref{wtbsq} 
    corresponding to the initial data $ \mathbbm u_0:=(\eta_0,  \mathbf v_0 )\in X^s_\mu(\R^d)  $ satisfying the non-cavitation condition $\inf_{x \in \R^d} \left\{1+ \epsilon \eta_0(x)\right\} \ge h_0 $ for some $h_0\in (0,1)$,   where the solution is defined on its maximal time of existence and satisfying
the blow-up alternative:
\begin{align}\label{eq:BA}
    \text{If } \ \ T^\ast <\infty, \text{ then } \limsup_{t\nearrow T^\ast} \norm{  \mathbbm u(t)}_{X^s_\mu} = \infty.
\end{align}
Then there exists a positive time 
    \begin{equation}\label{Td}
       T_d=    \left(c\ell_0^2  (\beta_\ast \mu)^{- \frac {1-\delta_d}4} \epsilon \norm{ \mathbbm u_0}_{X^s_\mu} \right)^{ -\frac {4}{4-d+2\delta_d}}    
    \end{equation}
    such that $T\ast > T_d$ and  
    \begin{align*}
        \sup_{t\in [0,T_d]} \norm{  \mathbbm u(t)}_{X^s_\mu} \le 4\ell_0  \norm{  \mathbbm u_0}_{X^s_\mu},
    \end{align*}
    where $\ell_0=\sqrt{ p_0/h_0}$ with $p_0:=1 + \epsilon \norm{ \mathbbm u_0}_{L^\infty_x}$,  $c=c(s)$ is a positive constant, $ \delta_1=0$ and $ \delta_2>0$ sufficiently small.
\end{lemma}

\begin{proof}
Define
    \begin{align}\label{Ttilde}
        \Tilde{T} := \sup\left\{ T \in (0,T^\ast) :    \sup_{t\in [0,T]} \norm{  \mathbbm u(t)}_{X^s_\mu} \le  4\ell_0 \norm{  \mathbbm u_0}_{X^s_\mu} \right\}.
    \end{align}
To obtain the result, we will argue by contradiction by assuming that $\Tilde{T} < T_d$.
    
Recall
$        p_d(t) :=  1 + \epsilon \norm{ \mathbbm u(t)}_{L^\infty_x}
$   and hence $p_d(0) = p_0=h_0\ell_0^2$.
We claim that the following hold for all $0\le t\le  \min(T_d,\tilde T)$:
        \begin{align}\label{Apriori1}
 \inf_{x \in \R} \left\{ 1+ \epsilon\eta(x,t) \right\}& \geq \frac{h_0}{2}    ,
   \\
   \label{Apriori2}
    p_d(t) &  \leq 2p_0   ,
    \end{align}
  The first inequality is the non-cavitation condition \eqref{noncav} with $ h_0/2$ replacing $h_0$.  
      Using \eqref{Apriori1} and \eqref{Apriori2} in \eqref{E-derv},   we get 
\begin{equation*}
\frac d{dt}  \mathcal E_s( \mathbbm u(t))  \le     4 c_0 \epsilon       \ell_0^2  q_d(t)\mathcal E_s( \mathbbm u(t))
\end{equation*}
 for all $t \in [0,\tilde T]$.
Gr\"onwall's inequality implies
\begin{equation}\label{ysoln1}
 \mathcal E_s( \mathbbm u(t))   \le   \exp\left[ 4 c_0 \epsilon     \ell_0^2   \int_0^t q_d(\tau) \, d\tau\right] \mathcal E_s( \mathbbm u_0).
\end{equation}
Combining this with \eqref{EX-eq},  \eqref{qdest}--\eqref{sigmaT}, \eqref{Ttilde} and \eqref{Td},  we obtain 
\begin{align*}
  \norm{  \mathbbm u(t)}_{X^s_\mu}   & \le  \ell_0 \exp\left[2 c_0  \ell_0^2   \epsilon   \int_0^{t} q_d(\tau) \, d\tau\right]  \norm{  \mathbbm u_0}_{X^s_\mu} 
  \\
  & \le  \ell_0 \exp\left[ c_1     \left ( \ell_0^2 \epsilon \sigma_d(\mu, \beta, T) \|  \mathbbm u_0\|_{ X_\mu^{s }}+ \ell_0^4  \epsilon^2  \sigma^2_d(\mu, \beta, T)  \norm{  \mathbbm u_0}^2_{X^s_\mu}  \right)\right]  \norm{  \mathbbm u_0}_{X^s_\mu} 
  \\
  & \le 2 \ell_0  \norm{  \mathbbm u_0}_{X^s_\mu} 
\end{align*}
by choosing the constant $c$ in \eqref{Td} larger than  $ c_1/ \ln 2$.
Then by continuity, we deduce that there exists some $\tau \in  [\widetilde T, T^\ast)$ such that
      $\norm{  \mathbbm u(\tau)}_{X^s_\mu} \le 4  \ell_0  \norm{  \mathbbm u_0}_{X^s_\mu} $.   This contradicts the
definition of $\widetilde T$. Therefore,  $ T_d \le \widetilde T $, where $ T_d$ is as in \eqref{Td}. This concludes the proof of the Theorem.

    It remains to prove the claims \eqref{Apriori1}--\eqref{Apriori2}.
      To prove \eqref{Apriori2}, we use the fundamental theorem of calculus to write
\begin{align*}
1+\epsilon \eta(x,t) + \epsilon  \mathbf v(x,t) &= 1+ \epsilon  \eta_0(x) + \epsilon \mathbf v_0(x)+   \epsilon \int_0^t\left[ \partial_t \eta(x, \tau)\ +   \partial_t \mathbf v(x,  \tau) \right]d\tau .
\end{align*}
Thus, 
\begin{align}\label{pde}
p_d(t):= 1 + \epsilon \norm{ \mathbbm u(t)}_{L^\infty_x}& \le p_0 + \epsilon  \left(  \|  \eta_t\|_{ L^1_{T_d}L^\infty_x} +  \| \mathbf v_t\|_{  L^1_{T_d} L^\infty_x}  \right).
\end{align}

On the other hand, we can use \eqref{wtbsq} to estimate
\begin{align*}
\|\eta_t (t)\|_{L^\infty_x} + \|\mathbf v_t (t)\|_{L^\infty_x} &\le  \| \mathcal K_{\mu} \mathbf \nabla \mathbf v\|_{L^\infty_x}  +   \| \nabla \eta\|_{L^\infty_x}+ \epsilon \left( \| \eta\|_{L^\infty_x}+      \| \mathbf v\|_{L^\infty_x}    \right)   \left(   \| \nabla \eta\|_{L^\infty_x}  + \| \mathbf \nabla v\|_{L^\infty_x}   \right)
\\
&\le p_d(t) q_d(t)   , 
\end{align*}
and hence
\begin{equation}\label{etavt-bd}
  \epsilon  \left(  \|  \eta_t\|_{ L^1_{T_d}L^\infty_x} +  \| \mathbf v_t\|_{  L^1_{T_d} L^\infty_x}  \right)\le   \epsilon 
 \| p_d\|_{ L^\infty_{T_d}} \| q_d\|_{ L^1_{T_d}} . 
\end{equation}
Using \eqref{qdest}--\eqref{sigmaT}, \eqref{Ttilde} and  \eqref{Td}, we obtain
\begin{equation}\label{q-bd}
\begin{split}
    \epsilon  \| q_d\|_{ L^1_{T_d}}  
&\le  c_1 \left[  \epsilon \sigma_d(\mu, \beta, T)\|  \mathbbm u_0\|_{ X_\mu^{s }}+ \ell_0^2 \epsilon^2  \sigma^2_d(\mu, T)   \|  \mathbbm u_0 \|^2_{ L^\infty_T X_\mu^{s }  }\right] 
\\
&\le  
  (2\ell_0)^{-2}
 \end{split}
\end{equation}
by choosing the constant $c$ in \eqref{Td}  larger than $ 4c_1$.

Consequently,  combining \eqref{pde}--\eqref{q-bd}, we get
\begin{align*}
p_d(t) 
&\le  p_0 + 
\frac 12 \| p_d\|_{ L^\infty_{T_d}}  \qquad \Rightarrow \qquad \| p_d\|_{ L^\infty_{T_d}} \le 2p_0
\end{align*}
which proves \eqref{Apriori2}.

Finally,  we prove \eqref{Apriori1}.  By the fundamental theorem of calculus,
 \eqref{etavt-bd},  \eqref{q-bd} and \eqref{Apriori2},  we have
\begin{align*}
1+ \epsilon \eta(x,t) = 1+ \epsilon  \eta_0(x) + \epsilon \int_0^t \partial_t \eta(x,\tau)\,d\tau
&\ge  h_0 -  \epsilon   \|\eta_t\|_{ L^1_{  T_d}L^\infty_x}
\\
&\ge h_0 -  \epsilon \| p_d\|_{ L^\infty_{T_d}} \| q_d\|_{ L^1_{T_d}}  
\\
&\ge h_0 -  2  p_0   \cdot (2\ell_0)^{-2} 
\\
&=h_0 -\frac{ h_0}2=h_0.
\end{align*}
 This yields \eqref{Apriori1}. 

\end{proof}

\subsection{Proof of Theorem \ref{thm-wtbq}  }

\subsubsection*{Existence of solution and Continuity
of the flow map}
With the a priori estimate in Lemma  \ref{lm-ap-wtbsq} at hand, the complete proof of existence of solutions would
then follow from a standard compactness argument implemented on a regularized
version of the system.  The argument can for instance be found in \cite[Section 4.1]{PSST21}.  Continuity
of the flow map follows from an application of the Bona-Smith argument \cite{BS1975}. We omit the details and refer the interested reader to \cite{LPS2018}.

\subsubsection*{Uniqueness  of solution  }

Let $\mathbbm u_1:=(\eta_ 1 ,  \mathbf v_1)$ and $\mathbbm u_2:=(\eta_2,   \mathbf v_2)$ be two pair of solutions of the system
\eqref{wtbsq} in the class \eqref{U-solnclass} with initial data $\mathbbm u_ 1 (0)=\mathbbm u_{1, 0}$,  $\mathbbm u_2(0)=\mathbbm u_{2, 0}$.  Let
\[
\zeta := \eta_1 - \eta_2,\qquad \mathbf  w :=  \mathbf v_1 -  \mathbf v_2.
\]
Then $  \mathbbm z:= \mathbbm u_1- \mathbbm u_2=(\zeta ,  \mathbf  w)$ solves the system
\begin{align}\label{wtbsq-dif} 
    \left\lbrace
    \begin{array}{l}
            \zeta_t   +  \mathcal K_{\mu, \beta}(D) \nabla \cdot \mathbf w =- \epsilon \nabla \cdot (\zeta \mathbf v_1 + \eta_2 \mathbf  w)\\
        \mathbf  w_t +  \nabla \zeta =- \frac{\epsilon}{2} \nabla \left(  \mathbf  w \cdot  \mathbf v_1+  \mathbf v_2 \cdot  \mathbf  w \right) .
 \end{array}\right.
\end{align}
with initial data $V (0)=V_0:=U_{1, 0}-U_{2, 0}$.

Define
\begin{equation}\label{Edif}
\mathcal E_0( \mathbbm z (t))=\frac{1}{2}\int_{\R^d} \zeta^2+
| \sqrt{\mathcal K_{\mu, \beta}}  \mathbf  w|^2 + \epsilon  \eta_1 | \mathbf  w|^2\, dx.
\end{equation}

Let $s>d/4+5/2$ and 
\begin{align*}
  \sigma(t)&= \sum_{j=1}^2  \norm{ \mathbbm u_j(t)}_{X^s_\mu}  \quad \text{with }\quad   \sigma(0)=\sigma_0=\sum_{j=1}^2  \norm{ \mathbbm u_{j,0}}_{X^s_\mu}.
\end{align*}

We claim that the following estimate hold:
\begin{equation}\label{E0-der}
\frac{ d}{dt}  \mathcal E_0( \mathbbm z(t)) \lesssim     \epsilon  h_0^{-1}  (1+\sigma(t))^2  \mathcal E_0(\mathbbm u z(t)).
\end{equation}
Suppose for the moment that \eqref{E0-der} holds. By \eqref{U-solnbound}, we have
\begin{equation}\label{sigmabd}
      \sup_{0 \le t \le T_{d}}  \sigma (t)\lesssim     \sigma_0 ,
    \end{equation}
    where $ T_{d} = T_d( \epsilon,\mu, \beta) $ as in \eqref{btime}.
By Gr\"onwall's inequality, and \eqref{sigmabd}, we have for $0\le t \le T_d$,
\begin{equation}\label{Vesoln}
 \mathcal E_0( \mathbbm z(t))   \le  \exp\left[ c_0 \epsilon     h_0^{-1}  T_d  (1+ \sigma_0)^2 \right]\mathcal E_0(\mathbbm  z_0) 
\end{equation}

Form H\"{o}lder’s inequality, the non-cavitation condition and \eqref{sigmabd}, we have
 \begin{align}\label{ceov-E0}
 h_0 \|\mathbbm z(t)\|_{X_\mu}^2 \lesssim \mathcal E_0(\mathbbm z(t)) \lesssim (1+  \epsilon \sigma_0)\|\mathbbm z(t)\|_{X_\mu}^2,
\end{align}
where $X_\mu:=X_\mu^0$.
Combining this with \eqref{Vesoln}, we obtain
\begin{align*}
 \sup_{0 \le t \le T_{d}}   \norm{ \mathbbm z (t)}_{X_\mu}   & \lesssim  \sqrt{ \frac{ 1+  \epsilon \sigma_0}{h_0}} \exp \left[   c_0 \epsilon     h_0^{-1}  T_d    (1+ \sigma_0)^2 \right] \|\mathbbm z_0\|_{X_\mu} , 
\end{align*}
and hence the uniqueness result in Theorem \ref{thm-wtbq} follows by choosing $\mathbbm u_{1, 0}=\mathbbm u_{2, 0}=\mathbbm u_0$.

It remains to prove \eqref{E0-der}.
We give the proof only for $d=1$ (the proof for $d=2$ is similar). In view of \eqref{ceov-E0},  it suffices to show
\begin{equation}\label{E0-derv1}
\frac{ d}{dt}  \mathcal E_0( \mathbbm z(t)) \lesssim     \epsilon (1+\sigma(t))^2 \|\mathbbm z(t)\|_{X_\mu}^2.
\end{equation}

For brevity, we write $\mathcal K:=\mathcal K_{\mu, \beta}$.
Differentiating $\mathcal E_0(V(t))$ with respect to $t$,  using the system \eqref{wtbsq-dif} for $d=1$ and applying integration by parts, we obtain
\begin{align*}
\frac{ d}{dt} \mathcal E_0(\mathbbm z(t)) &=-\int_{\R} (  \zeta  \mathcal K w_x+ \sqrt  \mathcal K  \zeta_x  \sqrt   \mathcal K w ) \, dx+  \epsilon    \int_{\R}  \zeta_x (\zeta v_1 + \eta_2w) \, dx  + \frac \epsilon 2\int_{\R} \sqrt  \mathcal K w_x \sqrt  \mathcal K w (v_1 + v_2 ) \, dx 
\\ & - \frac \epsilon  2  \int_{\R}  \mathcal K  \partial_x v_2 w^2  \, dx- \frac{ \epsilon^2}  2  \int_{\R}   \partial_x( \eta_2 v_2) w^2  \, dx- \epsilon   \int_{\R} \eta_2  w \zeta_x   \, dx + \frac{ \epsilon^2}  2  \int_{\R}   \partial_x( \eta_1w) w(v_1 + v_2)  \, dx
\\
&:=\sum_{j=1}^{7} I_j(t).
\end{align*}
Integration by parts yields $I_1=0$,
\begin{align*}
I_2+I_6&= \epsilon    \int_{\R}  \zeta_x \zeta v_1 \, dx =  - \frac  \epsilon 2 \int_{\R}  \zeta^2  \partial_x v_1\, dx 
\\
I_7&= - \frac{ \epsilon^2}  2  \int_{\R}    \eta_1w w_x (v_1 + v_2)  \, dx- \frac{ \epsilon^2}  2  \int_{\R}    \eta_1w^2 \partial_x(v_1 + v_2)_x  \, dx
\\
&= \frac{ \epsilon^2}  4  \int_{\R}   w^2   \partial_x(\eta_1(v_1 + v_2))  \, dx- \frac{ \epsilon^2}  2  \int_{\R}    \eta_1w^2  \partial_x(v_1 + v_2)  \, dx
\end{align*}

 By H\"{o}lder's inequality and Sobolev embedding, 
\begin{align*}
|I_4|+|I_2+I_6|&\le     \frac \epsilon 2   \norm{\mathcal K \partial_x v_2}_{L^\infty} \norm  {u}^2_{L^2}+ \frac \epsilon 2   \norm{\partial_x v_1}_{L^\infty} \norm  {\zeta}^2_{L^2}
\\
&\lesssim \|v_2\|_{H^s} \norm  {u}^2_{L^2} + \|v_1\|_{H^s} \norm  {\zeta}^2_{L^2} 
\\
&\lesssim \epsilon \sigma(t) \|\mathbbm z(t)\|^2_{X_\mu}
\end{align*}
and 
\begin{align*}
 I_{31}+I_{32} &\le     \frac \epsilon 2   \norm{\mathcal K \partial_x (v_1- v_2)}_{L^\infty} \norm  {u}^2_{L^2}+ \frac \epsilon 2   \norm{\partial_x v_1}_{L^\infty} \norm  {\zeta}^2_{L^2}
\\
&\lesssim \|v_2\|_{H^s} \norm  {u}^2_{L^2} + \|v_1\|_{H^s} \norm  {\zeta}^2_{L^2} 
\\
&\lesssim \epsilon \sigma(t) \| \mathbbm z(t)\|^2_{X_\mu}
\end{align*}

Similarly, 
\begin{align*}
|I_2+I_7|& \le   \epsilon^2  \norm{\partial_x(\eta_2v_2)}_{L^\infty}  \norm  {u}^2_{L^2} +    \epsilon^2  \left[\norm{\partial_x(\eta_1(v_1 + v_2))}_{L^\infty} +[\norm{\eta_1  \partial_x(v_1 + v_2)}_{L^\infty} \right] \norm  {u}^2_{L^2}  \\
& \le  \epsilon^2  \sum_{j=0}^1 \left(  \norm{ \partial^{j}_x \eta_1 }_{L^\infty} + \norm{\partial^{j}_x  \eta_2 }_{L^\infty} \right)  \left(  \norm{\partial^{1-j}_x v_1 }_{L^\infty} + \norm{ \partial^{1-j}_x v_2 }_{L^\infty} \right) \norm  {u}^2_{L^2}  \\
&\lesssim \epsilon^2 \sigma^2(t) \|\mathbbm z(t)\|^2_{X_\mu}
\end{align*}

By integration by parts, and using $w=v_1-v_2$,  we write
\begin{align*}
I_3&=  \frac \epsilon 2\int_{\R}  \mathcal K w_x \cdot w(v_1 + v_2 )) \, dx =   \frac \epsilon 2 \int_{\R}  \mathcal K w_x \cdot w^2 \, dx  +  \frac \epsilon 2 \int_{\R}  \mathcal K w_x  \cdot w v_2 \, dx
\\
&=  \frac \epsilon 2 \int_{\R}  \mathcal K w_x \cdot w^2 \, dx  +  \frac \epsilon 2 \int_{\R}  \mathcal K w  \cdot w \partial_x v_2 \, dx +\frac \epsilon 2 \int_{\R}  \mathcal K w  \cdot w_x  v_2 \, dx
\\
&:=  I_{31}+I_{32}+I_{33}, 
\end{align*}
Then
\begin{align*}
| I_{31}|\le     \frac \epsilon 2   \norm{\mathcal K w_x }_{L^\infty} \norm  {u}^2_{L^2} &\lesssim \left(\|v_1\|_{H^s} + \|v_2\|_{H^s} \right) \norm  {u}^2_{L^2} \\
&\lesssim \epsilon \sigma(t) \|\mathbbm z(t)\|^2_{X_\mu}.
\end{align*}
The terms of $I_{32}$ an $I_{33}$ are estimated in \cite{P2022} (see proof of Proposition 4.1; specifically, see estimates for the terms $II_4$ and $A^{(2)}_4$ therein). These terms also satisfy the bound 
\begin{align*}
| I_{32}|+ | I_{33}|&\lesssim  \epsilon \sigma(t)\|\mathbbm z(t)\|^2_{X_\mu}.
\end{align*}

\section{Proof of Lemma \ref{lm-mderv-est} }\label{sec5}

The phase function is
$$m_\beta(r)=r \angles{  \sqrt{ \beta} r}   \tau (r), \quad \text{with} \ \ \tau (r)= \sqrt{\tanh(r)/r}.$$ 
Let 
$s(r):= \text{sech}(r).$
Clearly, 
$$ \tau(r)\sim  \angles{r}^{-1/2}, \qquad  s(r)
\sim e^{-r}.$$

\subsubsection{
Proof of  \eqref{m-1stderv-est}} 
Since 
\begin{align*}
\tau'(r)&= \frac 1{2 r } \left(  \tau^{-1}s^2-\tau\right)
\end{align*}
we have
\begin{align*}
m_\beta'(r) &= \langle\sqrt{\beta}\,r\rangle(\tau+r\tau') + \beta r\langle\sqrt{\beta}\,r\rangle^{-1}(r\tau) \\
&= \frac{1}{2}\langle\sqrt{\beta}\,r\rangle\left(\tau+\tau^{-1}s^2\right) + \beta r^2\langle\sqrt{\beta}\,r\rangle^{-1}\tau \\
&\sim \langle\sqrt{\beta}\,r\rangle\langle r\rangle^{-1/2}.
\end{align*}

\subsection{
Proof of  \eqref{m-2ndderv-est}} 

Letting
\begin{align*}
\alpha_\beta(r)&:=\angles{ \sqrt{\beta} r}^{-2}   \left[ 1+  \tau^{-2}s^2+ \angles{\sqrt{\beta}  r}^{-2} \right], \qquad \gamma(r):=\frac{e^{2r}- e^{-2r}-4r} {4r^2},
\end{align*}
we can write (see \cite {PSST21}, \cite {DST20} for the detail),
\begin{align*}
m''_\beta(r)= 4^{-1}r \angles{\sqrt{\beta}  r} \tau(r)  h(r)  f_\beta(r),
\end{align*}
where
\begin{align*}
 \sigma(r) &:= 4  s^2 +  \tau^{-4} \gamma^2 s^4,
  \qquad
   f_\beta(r)
= 4 \beta \frac{\alpha_\beta(r)}{\sigma(r) }-1.
\end{align*}
Now, since
\begin{align}\label{ABE-est}
\gamma(r)\sim 
r\angles{r}^{-3}e^{2r} , \qquad  \alpha_\beta(r) 
\sim \angles{ \sqrt{\beta} r}^{-2} ,  \qquad  
  \sigma(r) 
\sim  \angles{r}^{-2} 
\end{align}
we have
\begin{align*}
|m_\beta''(r)| \sim r \angles{\sqrt{\beta}  r}  \angles{r}^{-\frac 52}  | f_\beta(r)|.
\end{align*}

Therefore, \eqref{m-2ndderv-est} reduces to establishing the lower and upper bound estimates
\begin{equation}
\label{fbest}
c_1 |3\beta-1| \le  | f_\beta(r)| \le c_2 \quad \forall \  r\ge 0,
\end{equation}
for some absolute positive constants $c_1$ and $c_2$,  whenever $\beta=0$ or  $\beta >1/3$.
Clearly, $| f_0(r)| =1$  for all $r\ge 0$, and it is also easy to see that the upper bound estimate in \eqref{fbest} holds for all $\beta\ge0$.
The lower bound estimate in 
\eqref{fbest} is given in the following Lemma.

\begin{lemma}[Positivity threshold for $f_\beta$]
\label{lm:fbeta-threshold}
Let $\gamma, \sigma, \alpha_\beta, f_\beta$ be as defined above. Then we have the following:

\begin{enumerate}[(i)]
\item \label{fbeta-a}  For every $\beta>1/3 $, we have the lower bound estimate
\begin{equation}\label{F-lbd}
f_\beta(r)\ \ge  \min(  3\beta-1, 1/2)  \qquad \forall \ r\ge0.
\end{equation}

\item \label{fbeta-b}  If $0<\beta\le 1/3$, then   
\begin{equation}\label{F-nolbd}
\inf_{r\ge0}|f_\beta(r)|=0
\end{equation}
and
\begin{equation}\label{F13}
    f_{1/3}(r)\ \ge\ 0 \qquad \forall \ r\ge 0.
\end{equation}

\end{enumerate}

\end{lemma}
We give the proof of Lemma \ref{lm:fbeta-threshold} in the following subsection.

\subsection {Proof of Lemma \ref{lm:fbeta-threshold} }
We need the following properties for $f_\beta(r)$:
\begin{equation}\label{fbeta-limit}
            \lim_{r \to 0^+} f_\beta(r) =  3\beta - 1 \quad  \& \quad  \lim_{r \to \infty} f_\beta(r)   = 3 \qquad \forall \ \beta \ge 0.
\end{equation}
The first limit follows from the fact that 
  $\alpha_\beta(r) \to 3$ and  $\sigma(r) \to 4$  as $r \to 0^+$. The second limit follows from \eqref{ABE-est}. Indeed, 
        \[
        \lim_{r \to \infty} f_\beta(r) =  \lim_{r \to \infty}  4\beta \frac{\alpha_\beta(r)}{\sigma(r) }-1=\lim_{r \to \infty}  \frac{4\beta (\sqrt \beta r)^{-2}}{r^{-2} } - 1  = 3.
        \]

Let us rewrite $ f_\beta(r)$ by setting
$$ 
x(r):= \tau^{-2}s^2 =\frac{2r}{\sinh 2r} \in(0,1], \qquad y(\beta,r):=\langle\sqrt\beta\,r\rangle^{-2} =\frac 1{1+ \beta r^2 } \in(0,1].
$$
Since $\sinh 2r - 2r = 2r^2 \gamma(r)$, we have $1-x= rx \gamma(r). $ So we can write 
\begin{align*}
    \sigma(r)= \frac {4r^2s^2 +(1-x)^2}{r^2}, \qquad  \beta=\frac{1-y}{r^2 y},
\end{align*}
and hence
$$
f_\beta(r)=\frac{g(y( \beta, r))}{h(r)}-1,
$$
where $g(y)= 4(1-y)(1+x+y)$ and  $h(r)= 4r^2s^2+(1-x)^2.$

\subsubsection{ Proof of \eqref{fbeta-b}: Case $0<\beta<1/3$}
By \eqref{fbeta-limit}, 
$f_\beta(0)=3\beta-1<0$ while $f_\beta(r)\to3>0$ as $r\to\infty$.  Since $f_\beta$ is continuous on $[0,\infty)$, by the Intermediate Value Theorem there exists $r_0 \in(0,\infty)$ such that $f_\beta(r_0)=0$, so $\inf_{r\ge0}|f_\beta(r)|=0$, which is \eqref{F-nolbd}.

\subsubsection{Proof of \eqref{fbeta-b}: Case $\beta=1/3$}

The statement $f_{1/3}(r)\ \ge\ 0$ is equivalent to $g(y( 1/3, r))\ge h(r)$ for all $ r\ge 0$, and therefore we need to show 
\begin{equation}\label{G13}
   d(r):= g(y( 1/3, r))- h(r)\ \ge\ 0 \qquad \forall \ r\ge 0.
\end{equation}
At $\beta=1/3$, $y=\dfrac{3}{3+r^2}$, and hence 
$$d(r)= \frac {4r^2(1+x)  }{3+r^2} + \frac {12r^2  }{ (3+r^2)^2} -4r^2s^2-(1-x)^2.$$
Since $\lim_{r\rightarrow 0 } x(r)= 1$,  $\lim_{r\rightarrow \infty } x(r)= 0$ and $\lim_{r\rightarrow \infty } s(r)=0$,  we have
$$ \lim_{r\rightarrow 0 } d(r)=0 \quad \text{and} \quad \lim_{r\rightarrow \infty } d(r)=3. $$
Moreover,  $d(r)> 0$ for all $r>0$  (see Fig. 1 below).
Therefore, $f_{1/3}(r)\ \ge\ 0$ for all $ r\ge 0$.

\begin{figure}[htbp]
    \centering
    \includegraphics[width=\textwidth]{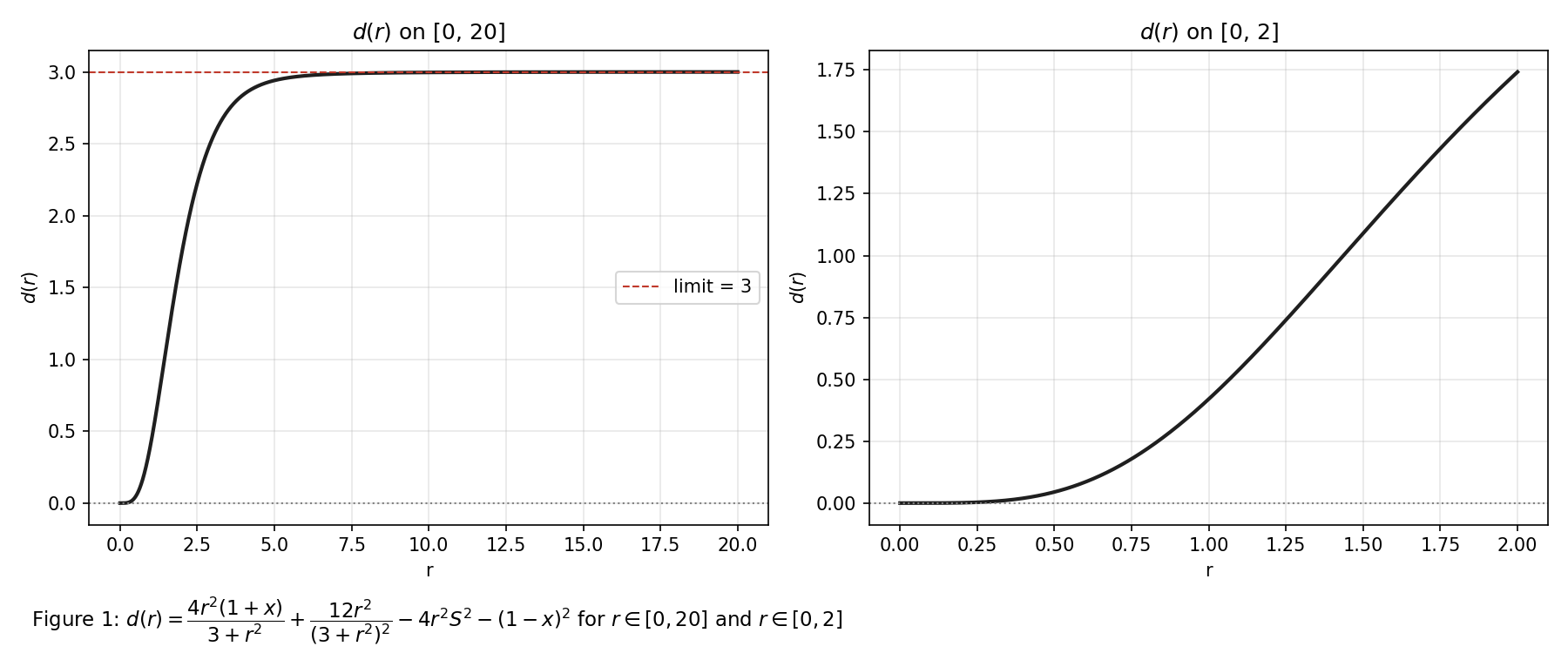}
    \label{fig:G_r}
\end{figure}

\subsubsection{Proof of \eqref{fbeta-a}: Case $\beta>1/3$} 

We write  $g(y)=4(1+x)-4x y-4y^2$.
Since $x(r)>0$ for all $r\ge0$, $g'(y)=-4x-8y<0$ for $y\ge0$, so $g$ is strictly decreasing in $y$. Since $y(\beta,r)$ is strictly decreasing in $\beta$, the composite $\beta \mapsto g(y(\beta,r))$ is strictly increasing for every fixed $r$; in particular, at $r=0$, we have $f_\beta(0)=3\beta-1$ which is increasing in $\beta$. 
Therefore, $f_\beta(r)$ increases strictly in $\beta$. This property can be combined with \eqref{F13} to deduce
$$ f_\beta(r)>f_{1/3}(r) \ge 0 \qquad \forall \ r\ge 0 , \quad \beta>1/3. $$
Since $f_\beta(r)$ is continuous on $[0, \infty)$, the Extreme Value Theorem implies that $c_\beta=\min_{r\ge0}f_\beta(r)$ is attained. Moreover, $c_\beta>0$, since $f_\beta$ never vanishes on $[0,\infty)$ for $\beta>1/3$.

To find $c_\beta$, we write the
Taylor expansions of $f_\beta(r)$.
Expanding $x(r)$ and $y(r)$, we have 
$$ x(r) = 1 - \tfrac{2}{3} r^2 + \tfrac{14}{45} r^4 + O(r^6), \qquad  y(r) = 1 - \beta r^2 + \beta^2 r^4 + O(r^6), $$
which imply
\begin{align*}
g(y) &= 4(1-y)(1+x+y)
=12\beta\, r^2 - \left(\frac{8\beta}{3} + 16\beta^2\right) r^4 + O(r^6),
\\ 
h(r) &= 4r^2s^2+(1-x)^2=4r^2 - \frac{32}{9} r^4 + O(r^6).
\end{align*}
Combining these expansions, we have
\[
\frac{g(y)}{h(r)}
= \frac{12\beta - \left(\tfrac{8\beta}{3}+16\beta^2\right) r^2 + O(r^4)}
{4 - \tfrac{32}{9} r^2 + O(r^4)}
= 3\beta + a r^2 + O(r^4),
\]
for some  $a=a(\beta)$. 
Matching coefficients gives $a = 2\beta - 4\beta^2$, and therefore 
\begin{equation*}
f_\beta(r)=\frac{g(y( \beta, r))}{h(r)}-1= (3\beta - 1) + 2\beta(1-2\beta)\, r^2 + O(r^4).
\end{equation*}
This shows that $r=0$ is a minimizer exactly while $1/3<\beta\le1/2$, in which case we have
$$
c_\beta=\min_{r\ge0}f_\beta(r)= f_\beta(0)=3\beta - 1.
$$
On the othere hand, 
for $\beta>1/2$, we use the strict monotonicity property of $f_\beta(r)$ to conclude that
$$
f_\beta(r) > f_{1/2}(r)\ge \frac 12 \qquad \forall \ r\ge 0,
$$
where we also used $f_{1/2}(r)\ge 1/2$ for all $r\ge 0$.
 Combining these two cases yield 
 $$ 
 f_\beta(r)\ \ge  \min(  3\beta-1, 1/2)  \qquad \forall \ r\ge0 \ \ \& \ \ \beta>1/3.
 $$

\vspace{8mm}

\subsection*{Acknowledgments}
 The authors acknowledge support from the Science Committee of the Ministry of Science and Higher Education of Kazakhstan (Grant No.  AP26194665).  The authors also acknowledge support from the Nazarbayev University Faculty Development Competitive
Research Grants Program 2025--2027 (Ref. 040225FD4730). The second author is grateful to Sigmund Selberg,  Didier Pilod and Nadia Taki
for fruitful discussions during the joint work in \cite{PSTT25} that motivated the present paper.

\end{document}